\pdfoutput=1

\documentclass[11pt]{amsart}
\usepackage[centertags]{amsmath}
\usepackage{amsfonts}
\usepackage{amssymb}
\usepackage{amsthm}
\usepackage{epsfig}
\usepackage{fullpage }
\usepackage{MnSymbol}
\usepackage{subfigure}
\usepackage{color}

\usepackage{pict2e} 

\newtheorem{theorem}{Theorem}
\newtheorem{proposition}{Proposition}
\newtheorem{lemma}{Lemma}
\newtheorem{corollary}{Corollary}
\theoremstyle{definition}
\newtheorem{definition}{Definition}

\theoremstyle{remark}
\newtheorem{remark}{Remark}
\newtheorem{example}{Example}

\begin{document}

\newcommand{\A}{\mathbb{A}}
\newcommand{\B}{\mathbb{B}}
\newcommand{\s}{\mathbb{S}}
\newcommand{\T}{\mathbb{T}}
\newcommand{\G}{\mathbb{G}}
\newcommand{\bS}{\boldsymbol{S}}

\newcommand{\overlay}{\uplus}
\newcommand{\ar}[1]{\overrightarrow{#1}}
\newcommand{\du}{\oast}


\title[Link diagrams and their parallels]{On the Seifert graphs of a link diagram and its parallels}

\author[S.~Huggett]{Stephen Huggett$^*$}
\author[I.~Moffatt]{Iain Moffatt$^\dagger$}
\author[N.~Virdee]{Natalia Virdee$^*$}

\thanks{
${\hspace{-1ex}}^*$School of Computing and Mathematics, University of Plymouth, PL4 8AA, Devon, UK.;  \\
${\hspace{.35cm}}$ \texttt{s.huggett@plymouth.ac.uk}, 
${\hspace{.35cm}}$ \texttt{natalia.virdee@plymouth.ac.uk}}

\thanks{
${\hspace{-1ex}}^\dagger$
Department of Mathematics and Statistics,  University of South Alabama, Mobile, AL 36688, USA;\\
${\hspace{.35cm}}$ \texttt{imoffatt@jaguar1.usouthal.edu}}



\date{}

\begin{abstract}
Recently, Dasbach, Futer, Kalfagianni, Lin, and Stoltzfus extended 
the notion of a Tait graph by associating a set of ribbon graphs (or 
equivalently, embedded graphs) to a link diagram. Here we focus on 
Seifert graphs, which are the ribbon graphs of a knot or link diagram that 
arise from Seifert states. We provide a characterization of Seifert 
graphs in terms of Eulerian subgraphs. This characterization can be 
viewed as a refinement of the fact that Seifert graphs are 
bipartite. We go on to examine the family of ribbon graphs that 
arises by forming the  parallels of a link diagram and determine how 
the genus of the ribbon graph of a $r$-fold parallel of a link 
diagram is related to that  of the original link diagram.
\end{abstract}

\maketitle


\section{Introduction}

There is a classical way to associate a (signed) plane graph, called 
the Tait graph, with the diagram of a link (see 
Subsection~\ref{ss.tg}). Tait graphs are a standard tool, and have 
found numerous applications in knot theory and graph theory. 
Recently, Dasbach,  Futer, Kalfagianni, Lin, and Stoltzfus, in 
\cite{Detal}, extended the idea of a Tait graph by associating a set 
of ribbon graphs (or equivalently, embedded graphs) with a link 
diagram.  The Tait graphs of a link diagram appear in this set of 
embedded graphs. One of the key advantages of Dasbach et al's idea of 
using non-plane graphs to describe knots is that it provides a way of 
encoding the crossing structure of a link diagram in the topology of 
the embedded graph (rather than by using signs on the edges). This 
idea is proving  to be very useful in knot theory and it has found 
many recent applications. These applications include applications to 
the Jones and HOMFLY-PT polynomials \cite{CP,CV,Ch1, Detal, Mo2, Mo3, 
VT10}; Khovanov homology \cite{CKS07, Detal}; knot Floer homology 
\cite{Lo08};    Turaev genus  \cite{Ab09,Lo08,Tu97}; 
Quasi-alternating links \cite{Wi09};  the coloured Jones polynomial 
\cite{FKP09}; the signature of a knot \cite{DL10}; the  determinant 
of a knot \cite{Detal, Detal2}; and hyperbolic knot theory 
\cite{FKP08}.

Here we are interested in the structure of the ribbon graphs of a 
link diagram, and of their underlying abstract graphs.  We are 
especially interested in the Seifert graphs of a link diagram which 
arise from the Seifert state of a link diagram. Seifert graphs are 
known to  be bipartite (see \cite{Cr04}, for example). However, as 
not all bipartite graphs arise from link diagrams, this does 
not, on its own, provide a characterization of the set of Seifert 
graphs of a link diagram. Here we provide a necessary and sufficient 
condition for a graph to be the Seifert graph of a link diagram. A 
well known result in graph theory states that a plane graph is 
bipartite if and only if its dual is Eulerian (see \cite{Bondy}, for 
example). Our characterization (in Theorem~\ref{t.sc}) of Seifert 
graphs will be stated in terms of the dual concept of Eulerian 
graphs.

We then go on to determine the operation on Tait graphs which 
corresponds to forming the $r$-fold parallel of a link diagram. We 
conclude by applying this result to finding the genus of the ribbon 
graph of a $r$-fold parallel of a link diagram in terms of the genus 
of the ribbon graph of the original link diagram.

\section{The graphs of knot and link diagrams}
In this section we provide an overview of the graphs and embedded 
graphs of a link diagram and we will describe their relation to each 
other. We will assume a familiarity with basic knot theory and graph 
theory.

\subsection{Tait graphs}\label{ss.tg}

Throughout this paper we will always assume that link diagrams 
consist of one component, that is, if $D\subset S^2$ is a link 
diagram, then each component of $S^2-D$ is a disc. The components of 
$S^2-D$ are called the {\em regions} of $D$. Since we prefer to work 
with graphs embedded in $S^2$ rather than $\mathbb{R}^2$, we will 
slightly abuse the language and refer to a graph (cellularly) embedded in 
$S^2$ as a {\em plane} graph.

Let $D\subset S^2$ be a link diagram. A {\em checkerboard colouring} 
of $D$ is an assignment of the colour black or white to each region 
of $D$ in such a way that adjacent regions have different colours. 
The {\em Tait sign} of a crossing in a checkerboard coloured link 
diagram is an element of $\{+,-\}$ which is assigned to the crossing 
according to the following scheme:
\[ \begin{array}{ccc}
\raisebox{1mm}{\includegraphics[height=1.2cm]{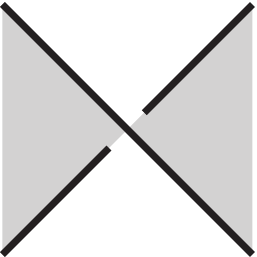}} & \quad 
\raisebox{6mm}{\text{or}} \quad& 
\raisebox{1mm}{\includegraphics[height=1.2cm]{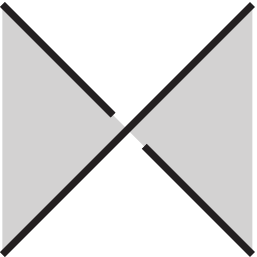}} \;.\\
$+$ & & $-$ 
\end{array}\]

A {\em Tait  graph} $\T(D)$ is a signed plane graph constructed from 
$D$ as follows: checkerboard colour the link diagram, place a vertex 
in each black region and add an edge between two vertices whenever 
the corresponding regions of $D$ meet at a crossing. Finally, weight 
each edge of the graph with the Tait sign of the corresponding 
crossing.

 
\begin{example} This example illustrates the construction of a Tait graph $\T(D)$ of a link diagram $D$.
\[\includegraphics[width=3cm]{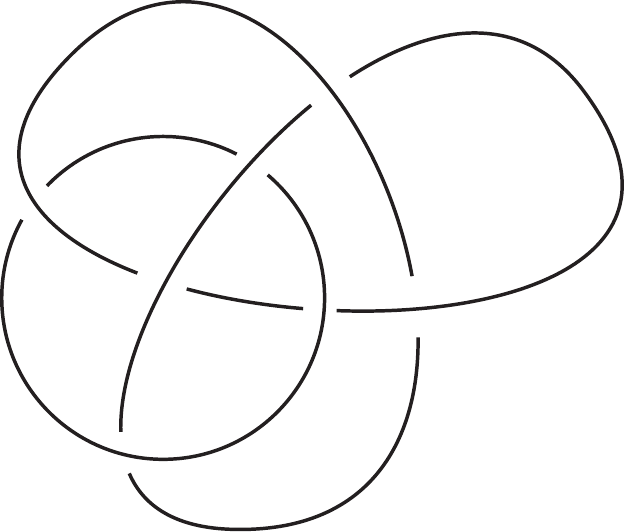} \raisebox{1cm}{\includegraphics[width=1cm]{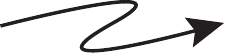}}
\includegraphics[width=3cm]{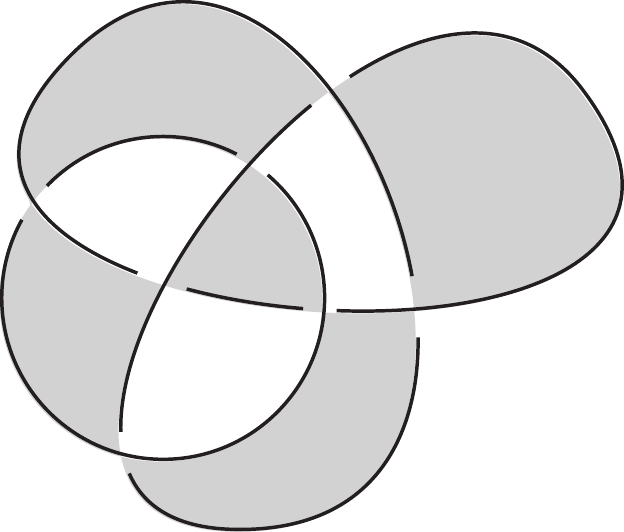}\raisebox{1cm}{\includegraphics[width=1cm]{arrow}}
\includegraphics[width=3cm]{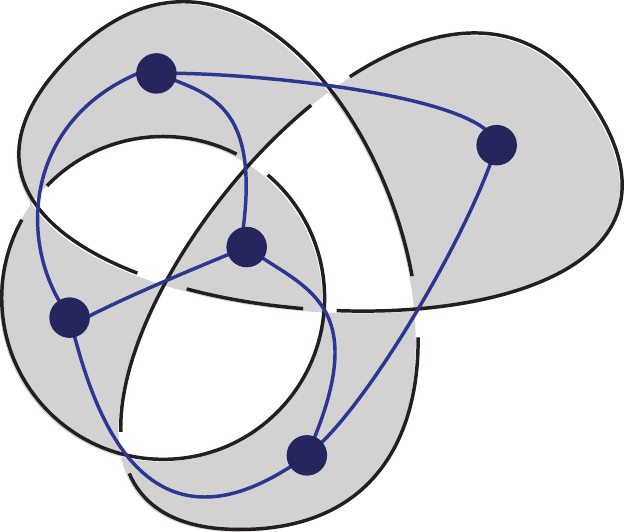}
\raisebox{1cm}{\includegraphics[width=1cm]{arrow}}
\includegraphics[width=3cm]{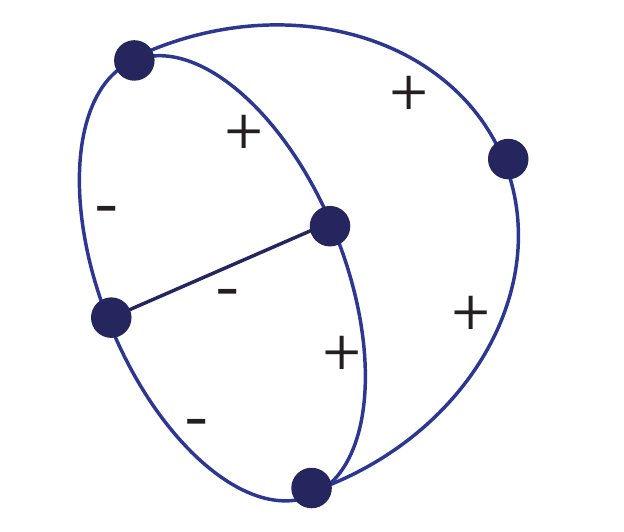}\]
\end{example}

Since there are two possible checkerboard colourings of  $D$, every 
diagram $D$ has exactly two  associated  Tait graphs. It is well 
known and easily seen that the two Tait graphs associated with a link 
diagram are duals of each other (we refer a  reader who is unfamiliar 
with duality  forward to Subsection~\ref{ss.dual} for its 
definition). Here it should be noted that duality acts on the edge 
weights by switching the sign. That is, if $e$ is an edge in a signed 
embedded graph $G$  with sign $m_e$, then the edge $e^*$ in $G^*$ 
that corresponds to $e$ will have sign $-m_e$.

\medskip

In addition to the Tait sign, we will also need to make use of the oriented sign of a crossing.  

Let $D$ be an oriented link diagram. The {\em oriented sign}  of a crossing of $D$ is an element of $\{+,-\}$ which is assigned to the crossing according to the following scheme:
\[ \begin{array}{ccc}
\raisebox{1mm}{\includegraphics[height=1.2cm]{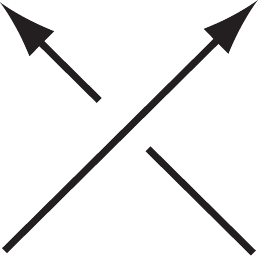}} & \quad& \raisebox{1mm}{\includegraphics[height=1.2cm]{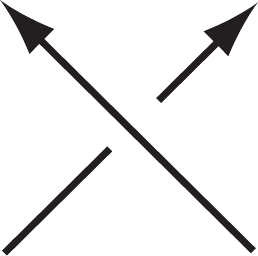}} \\
+ & & -
\end{array}.\]

An {\em bi-weighted Tait graph}, $\T_{\sigma}(D)$, of an oriented link 
diagram $D$ is a plane graph with edge weights in $\{+,-\}\times 
\{+,-\}$. The bi-weighted Tait graph is formed in the same way as the 
Tait graph except that the weight $(m_e,\sigma_e)$ is assigned to 
each edge $e$, where $m_e$ is the Tait sign of the crossing 
corresponding to $e$ and $\sigma_e$ is its oriented sign. 

Just as with Tait graphs, each link diagram gives rise to two 
bi-weighted Tait graphs. Moreover,  with an appropriate action of 
duality on the edge weights, these two bi-weighted Tait graphs are 
duals. If $e$ is an edge of weight  $(m_e, \sigma_e)$ in an 
embedded graph $G$, then the edge $e^*$ in $G^*$ that corresponds to 
$e$ will have sign $(-m_e, \sigma_e)$. With this action of duality on 
the edge weights it follows that if  $G$ and $H$ are the two 
bi-weighted Tait graphs associated with an oriented link diagram then 
$G^*=H$.

\subsection{Ribbon graphs and their representations}
An {\em embedded  graph} $G=(V(G),E(G)) \subset \Sigma$ is a graph 
drawn on a surface $\Sigma$  in such a way that edges only intersect 
at their ends. The  arcwise-connected components of $\Sigma 
\backslash G$ are called the {\em regions} of $G$. If each of the 
regions of an embedded graph $G$ is homeomorphic to a disc we say 
that $G$ is a {\em cellularly embedded graph}, and its regions are 
called {\em faces}. A {\em plane graph} is a graph that is cellularly 
embedded in the sphere.

Two embedded graphs, $G\subset \Sigma$ and  $G'\subset \Sigma'$  are 
said to be {\em equal}   if there is a homeomorphism from  $\Sigma$ 
to $\Sigma'$ that sends $G$ to $G'$. As is standard, we will often 
abuse notation and identify an embedded graph with its equivalence 
class under equality.

Tait graphs are edge-weighted graphs cellularly embedded in a 
sphere.  In this paper we are interested in higher genus analogues of 
Tait graphs. These are edge-weighted embedded graphs that arise from  
link diagrams. It will be particularly convenient to use the language 
of ribbon graphs to describe these cellularly embedded graphs.

\begin{definition}
A {\em ribbon graph} $G =\left(  V(G),E(G)  \right)$ is a (possibly non-orientable) surface with boundary represented as the union of two  sets of topological discs, a set $V (G)$ of {\em vertices}, and a set of {\em edges} $E (G)$ such that: 
\begin{enumerate}
\item the vertices and edges intersect in disjoint line segments;
\item each such line segment lies on the boundary of precisely one
vertex and precisely one edge;
\item every edge contains exactly two such line segments.
\end{enumerate}
\end{definition}

Ribbon graphs are considered up to  homeomorphisms of the surface that preserve the vertex-edge structure.  

A ribbon graph is said to be {\em orientable} if it is orientable as a surface. Here we will only consider orientable ribbon graphs. The {\em genus}, $g(G)$, of a ribbon graph is the genus of $G$ as a surface. In addition we let $p(G)$ denote the number of boundary components of $G$, $k(G)$ the number of its connected components, $e(G):=|E(G)|$ and $v(G):=|V(G)|$. 

Ribbon graphs are easily seen to be equivalent to cellularly embedded 
graphs. Intuitively, if $G$ is a cellularly embedded graph, a ribbon 
graph representation results from taking a small neighbourhood  of 
the cellularly embedded graph $G$. On the other hand, if $G$ is a 
ribbon graph, we simply sew discs into each boundary component of the 
ribbon graph to get the desired surface.

\begin{figure}
\begin{tabular}{ccccc}
\includegraphics[width=40mm]{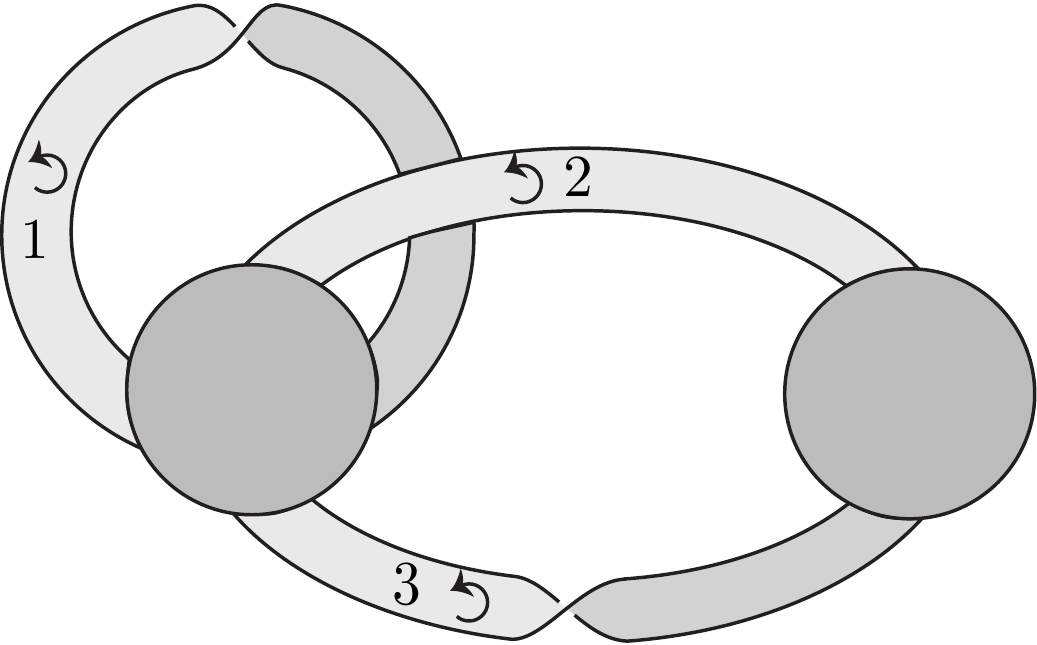} &\raisebox{9mm}{ = }& \includegraphics[width=40mm]{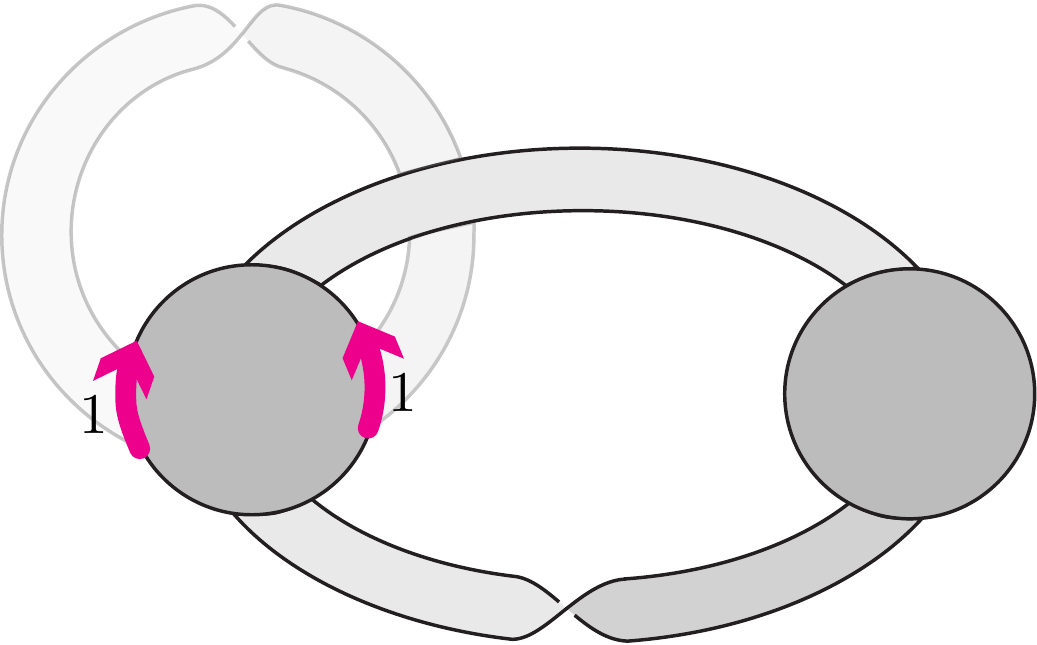} &\raisebox{9mm}{ = }& \includegraphics[width=40mm]{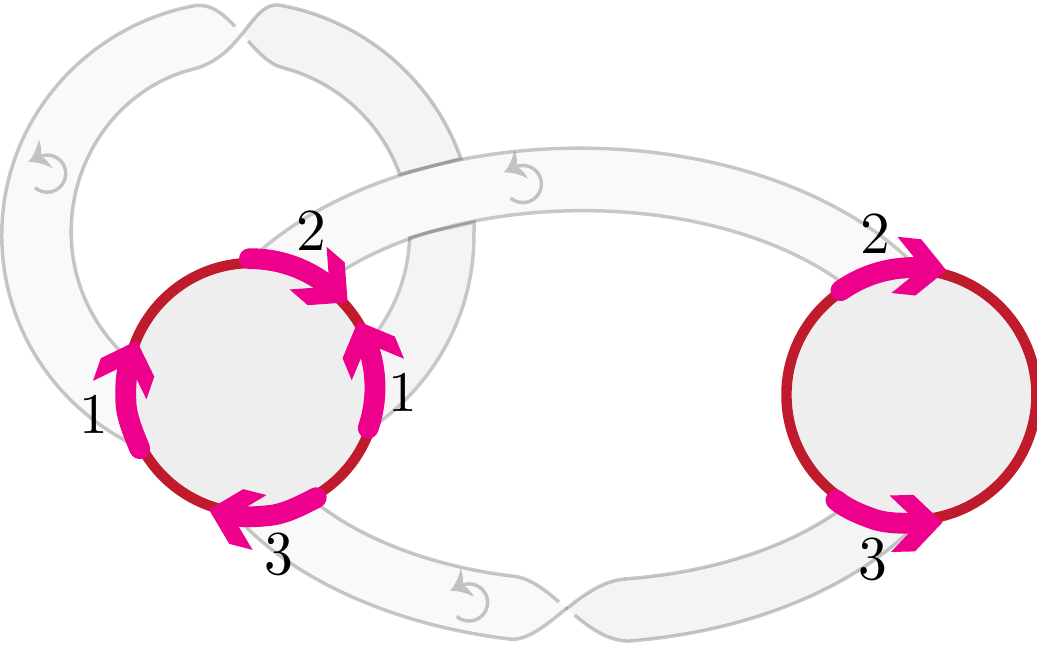} \\
(i)&&(ii)& &(iii)  
\end{tabular}
\caption{Realizations of a ribbon graph}
\label{fig.rgex}
\end{figure}

\bigskip

A {\em spanning ribbon subgraph} of $G$ is a ribbon graph $H$ which 
can be obtained from $G$ by deleting some edges. We will often use coloured arrows on the boundary of $H$ to record 
where these deleted edges were.
\begin{definition}
An {\em arrow-marked ribbon graph} $\ar{G}$ consists of a ribbon graph $G$ equipped with a collection of  coloured arrows, called {\em marking arrows}, on the boundaries of its vertices. The marking arrows are such that no marking arrow meets an edge of the ribbon graph, and    there are exactly two marking arrows of each colour.
\end{definition}

Two arrow-marked ribbon graphs are considered to be equivalent if one can be obtained from the other by reversing the direction of all of the marking arrows which belong to some subset of colours.

A ribbon graph can be obtained from an arrow-marked ribbon graph by adding edges in a way prescribed by the marking arrows, as follows: take a disc (this disc will form the new edge) and orient its boundary arbitrarily. Add this disc to the ribbon graph by choosing two non-intersecting arcs on the boundary of the disc and two marking arrows of the same colour, and then identifying the arcs with the marking arrows according to the orientation of the arrow. The disc that has been added forms an edge of a new ribbon graph. 
This process is illustrated in the diagram below, and an example of an arrow-marked ribbon graph and the ribbon graph it describes  is given in Figures~\ref{fig.rgex}(i) and (ii).
\[\includegraphics[height=12mm]{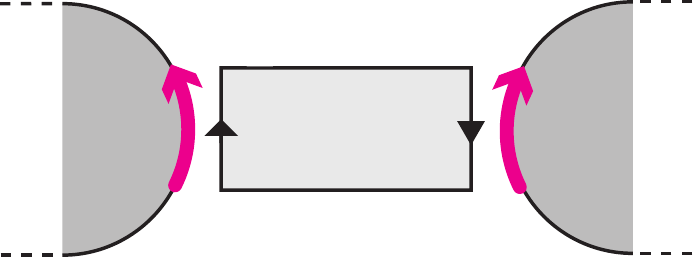} 
\raisebox{4mm}{\hspace{3mm}\includegraphics[width=12mm]{arrow}\hspace{3mm}}
  \includegraphics[height=12mm]{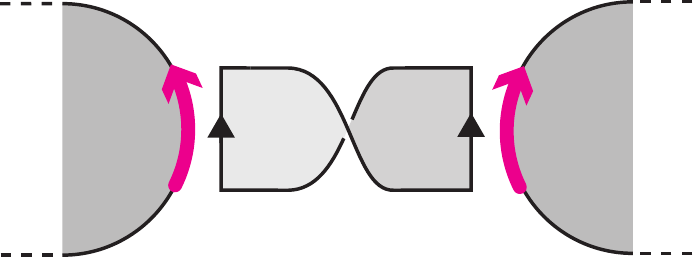} 
  \raisebox{4mm}{\hspace{3mm}\includegraphics[width=12mm]{arrow}\hspace{3mm}} \includegraphics[height=12mm]{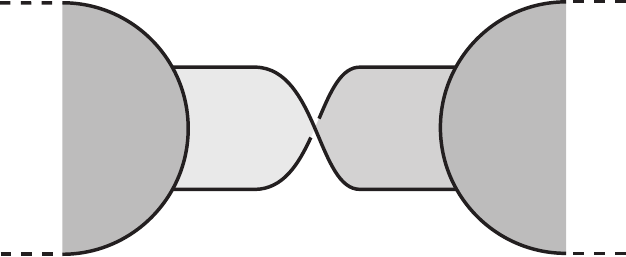} \]

The above construction shows that an arrow-marked ribbon graph 
describes a ribbon graph. Conversely, every ribbon graph can be 
described as an arrow-marked spanning ribbon subgraph. To see why 
this is, suppose that $G$ is a ribbon graph and $B \subseteq E(G)$. 
To describe $G$ as an arrow-marked ribbon graph $\ar{G- B}$, start by 
arbitrarily orienting each edge in $B$. This induces an orientation 
on the boundary of each edge in $B$. For each $e\in B$, place an 
arrow on each of the two arcs where $e$ meets vertices of $G$, the 
directions of these arrows following the orientation of the boundary 
of $e$. Colour the two arrows with $e$, and delete the edge $e$. This 
gives an arrow-marked ribbon graph $\ar{G- B}$. Moreover, the original 
ribbon graph $G$ can be recovered from $\ar{G- B}$ by adding 
edges to $\ar{G- B}$ as prescribed by the marking arrows. 

Note that if $G$ is a ribbon graph and $F$ is any spanning ribbon 
subgraph, then there is an arrow-marked ribbon graph  $\ar{F}$ which 
describes $G$.

\bigskip

Every  ribbon graph $G$ has a representation as an arrow-marked ribbon graph $\ar{V(G)}$, where the spanning ribbon subgraph consists of the vertex set  of $G$.   In such cases, to describe $G$ it is enough to record only the marked boundary cycles of the vertex set (to recover the vertex set, just place each cycle on the boundary of a disc).  Thus a ribbon graph can be presented as a set of cycles with marking arrows on them. In such a structure, there are exactly two marking arrows of each colour. 
Such a structure is called an {\em arrow presentation}. A ribbon graph can be recovered from an arrow presentation by regarding the marked cycles as boundaries of discs, giving an arrow-marked ribbon graph. To describe this more formally:
\begin{definition}
An  {\em arrow presentation} of a ribbon graph consists of a set of oriented (topological) circles (called {\em cycles}) that are marked with coloured arrows, called {\em marking arrows}, such that there are exactly two marking arrows of each colour.   
\end{definition} 
An example of a ribbon graph with its arrow presentation is given in 
Figure~\ref{fig.rgex}(i) and (iii).

Two arrow presentations are considered equivalent if one can be 
obtained from the other by reversing pairs of marking arrows of the 
same colour.

If weights are associated to the arrows, then an arrow 
presentation describes an edge-weighted ribbon graph.

\subsection{The ribbon graphs of a link diagram}\label{ss.rgld}
In \cite{Detal}, Dasbach et. al. extended the concept of a Tait graph 
by describing how a set of ribbon graphs can be associated to a link 
diagram. The ribbon graphs of a link diagram have proved to be very  
useful in knot theory. Here we are interested in certain special 
members of this set of ribbon graphs. In addition to the Tait graphs 
of a link diagram, we are particularly interested in the all-$A$, 
all-$B$ and Seifert ribbon graphs. We now define these ribbon graphs 
here.

\medskip

Let $D\subset S^2$ be a checkerboard coloured  link diagram.  Assign a unique label to each crossing of $D$. A {\em marked $A$-splicing} or a   {\em marked $B$-splicing} of a crossing $c$  is the replacement of the crossing with one of the following two schemes:

\begin{center}
\begin{tabular}{ccccc}
\includegraphics[height=1.2cm]{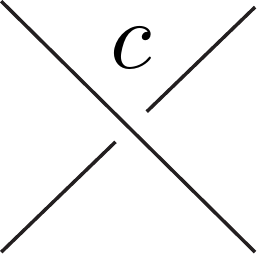}  & & \includegraphics[height=1.2cm]{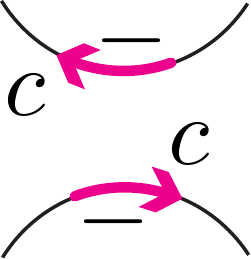} &&
\includegraphics[height=1.2cm]{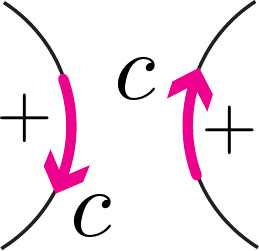} \\
 A crossing in $D$. &&  A marked $A$-splicing. && A marked $B$-splicing.
\end{tabular}
\end{center} 

Notice that we decorate the two arcs in the splicing with signed, 
coloured arrows, chosen to be consistent with an arbitrary 
orientation of the sphere $S^2$. The colour of the arrows is 
determined by the colour of the crossing, and the signs are 
determined by the choice of splicing.

A {\em state} $s$ of a link diagram is an assignment  of a marked $A$- or $B$-splicing to each crossing. Observe that a state is precisely an arrow presentation of a ribbon graph. We will  denote the ribbon graph corresponding to the state $s$ by $\G(D,s)$. We say that  $\G(D,s)$ is a {\em ribbon graph of the link diagram $D$}. We let $\mathcal{G}(D)$ denote the set of ribbon graphs of $D$, so that 
\[\mathcal{G}(D):= \{  \G(D,s)\; | \; s \text{ is a state of } D\}. \]
 $\mathcal{G}(D)$ is independent of the choice of checkerboard colouring used in its construction. 

If $D$ is an {\em oriented} link diagram we may repeat the above 
construction of signed ribbon graphs, but with weights 
$(m_c,\sigma_c)$, where $m_c$ is the Tait sign of the crossing $c$ 
and $\sigma_c$ is its oriented sign. If $s$ is such a state, we 
denote the ribbon graph by $\G_{\sigma}(D,s)$, and we will denote the 
set of ribbon graphs obtained in this way by $\mathcal{G}_{\sigma}(D)$.

\medskip

We will now describe some special elements of $\mathcal{G}(D)$.

\smallskip

{\bf Tait graphs:} Let $\T(D)$ denote a Tait graph of $D$. Then 
$\T(D)\in \mathcal{G}(D)$, and $\T(D)=\G(D,s)$, where the state $s$ is 
obtained either by choosing an $A$-splicing at each $-$ crossing and 
a $B$-splicing at each $+$ crossing, or by choosing a $B$-splicing at 
each $-$ crossing and an $A$-splicing at each $+$ crossing.

Similarly, the bi-weighted Tait graph $\T_{\sigma}(D)$ is an element of 
$\mathcal{G}_{\sigma}(D)$. Again we have 
$\T_{\sigma}(D)=\G_{\sigma}(D,s)$, where $s$ is constructed as for $\T(D)$.

Note that in this construction, the states are chosen so that the 
curves will always follow the black faces, or will always follow the white faces, of the 
checkerboard coloured link diagram.

As plane graphs and genus zero ribbon graphs are equivalent, we will
move freely between the realization of a Tait graph as a plane graph 
and as a genus zero ribbon graph. This should cause no confusion.

\smallskip

{\bf The all-$A$ and all-$B$ ribbon graphs:} The {\em all-$A$ ribbon 
graph} is defined by $\A(D)=\G(D,s)$, where $s$ is the state obtained 
by choosing an $A$-splicing at each crossing. Similarly, the {\em 
all-$B$ ribbon graph} is defined by $\B(D)=\G(D,s)$, where $s$ is the 
state obtained by choosing a $B$-splicing at each crossing.

\begin{example} This example illustrates the the construction of $\A(D)$.
\[\includegraphics[width=3cm]{l1} \quad \raisebox{1cm}{\includegraphics[width=1cm]{arrow}} \quad
\includegraphics[width=3cm]{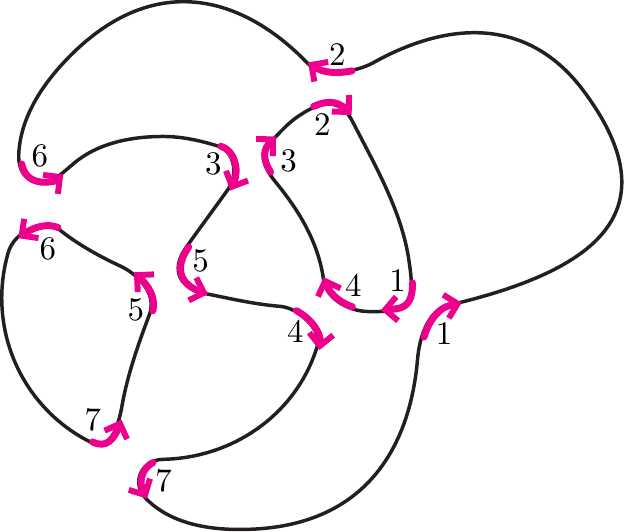} \quad \raisebox{1cm}{\includegraphics[width=1cm]{arrow}} \quad
\raisebox{4mm}{\includegraphics[width=5cm]{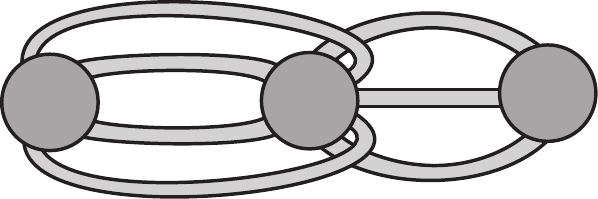}}
\]

\end{example}

\smallskip

{\bf The Seifert ribbon graph:} The Seifert ribbon graph $\s(D)$ is 
obtained by choosing the splicing that is consistent with the 
orientation, as in the following figure:

\begin{center}
\begin{tabular}{cccc}
\includegraphics[height=1.2cm]{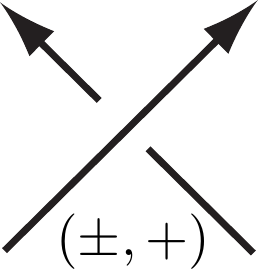}  
\raisebox{4mm}{\includegraphics[width=1.5cm]{arrow}}
\includegraphics[height=1.2cm]{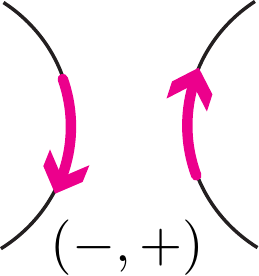}
&\hspace{1cm}\raisebox{7mm}{or}\hspace{1cm} & \includegraphics[height=1.2cm]{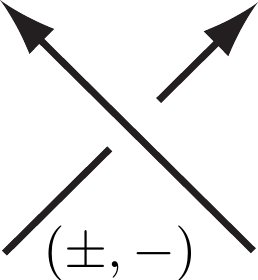}  
\raisebox{4mm}{\includegraphics[width=1.5cm]{arrow}}
\includegraphics[height=1.2cm]{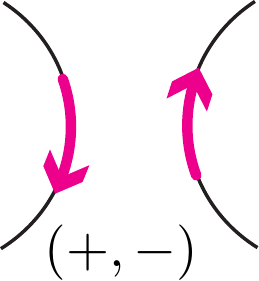} .
\end{tabular}
\end{center} 
Thus if $c$ is a crossing with oriented sign $+$, then choose the 
$A$-splicing at $c$, while if $c$ has oriented sign $-$ then choose 
the $B$-splicing at $c$.

Observe that the edge weights, $(+,-)$ and $(-,+)$, of a Seifert 
graph depend only on the oriented sign of the corresponding crossing 
in the original link diagram. This means that the construction of the 
Seifert graph with weights in $\{(+,-),(-,+)\}$ does not depend upon 
the checkerboard colouring, as one would expect.

\begin{example}
This example illustrates the the construction of $\s(D)$.
\[\includegraphics[width=3cm]{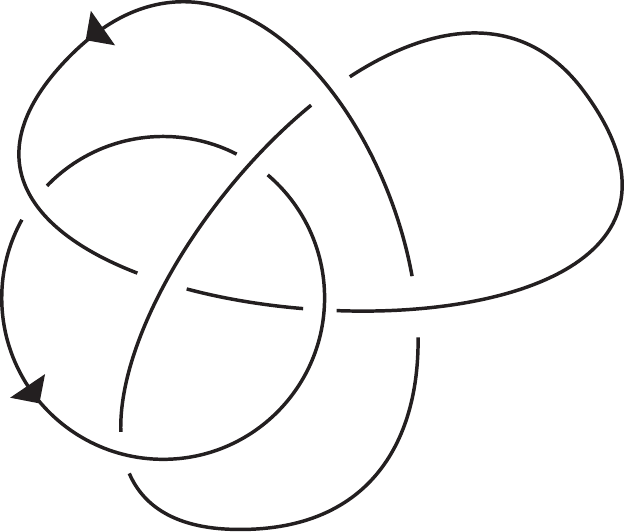} \quad \raisebox{1cm}{\includegraphics[width=1cm]{arrow}} \quad
\includegraphics[width=3cm]{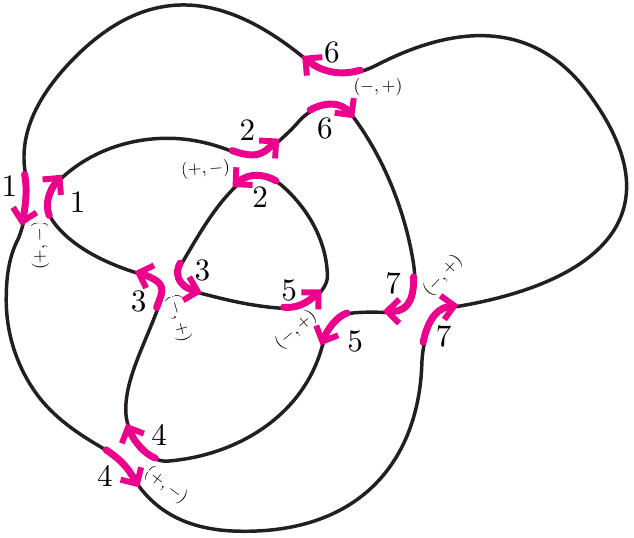} \quad \raisebox{1cm}{\includegraphics[width=1cm]{arrow}} \quad
\raisebox{4mm}{\includegraphics[width=5cm]{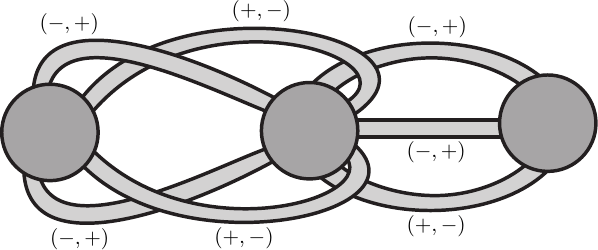}}
\]
\end{example}

\subsection{Dual graphs}\label{ss.dual}

The construction of the {\em dual}, $G^*$, of a cellularly embedded 
graph, $G\subset \Sigma$, is well known: form $G^*$ by placing one 
vertex in each face of $G$ and embed an edge of $G^*$ between two 
vertices whenever the faces of $G$ they lie in are adjacent. 
In particular, if $G$ has $k$  components, $G_1,\ldots, G_k$,  and is cellularly embedded in a surface, then each component of the graph is cellularly embedded in a connected component of the surface, and therefore duality acts disjointly on  components of the graph:  $(G)^* = G_1^* \cup \cdots \cup G_k^*$.

We will 
also need to form duals of non-cellularly embedded graphs. Since the 
properties of duality depend upon whether or not a graph is 
cellularly embedded, we will denote the dual of a not necessarily cellularly 
embedded graph by $G^{\du}$.
The embedded graph $G^{\du}$ is formed 
just as the dual of a cellularly embedded graph is formed, but by 
placing a vertex in each region of $G$, rather than each face. It is 
important to note that, in general,  
$(G^{\du})^{\du}\neq G$. 

There is a natural bijection between the edges of $G$ and the edges of $G^*$ (or of $G^{\du}$). We will generally use this bijection to identify the edges of $G$ and the edges of $G^*$. However, at times we will be working  with $G \cup G^*$ so to avoid confusion we will use $e^*$ to denote the edge of $G^*$ that corresponds to the edge $e$ of $G$, and adopt a similar convention for sets of edges.  

Observe that $G^*$ has a naturally cellular  embedding in $\Sigma$, and that there is a natural (cellular) immersion of $G \cup G^*$ where each edge of $G$ intersects exactly one edge of $G^*$ at exactly one point. We will call this immersion the {\em standard immersion} of $G \cup G^*$.

\medskip

Duality has  a particularly neat description in the language of 
ribbon graphs. Let $G=(V(G), E(G))$ be a ribbon graph, which we can 
regard $G$ as a punctured surface. By filling in the punctures using 
a set of discs denoted $ V(G^*) $, we obtain a surface without 
boundary $\Sigma$. The {\em dual} of $G$  is the ribbon graph $G^* = 
(V(G^*), E(G)) $.

Suppose now that $\ar{G}$ is an arrow-marked ribbon graph, so that 
$\ar{G}$ is a ribbon graph $G$ with  labelled arrows on its 
vertices. Then in the formation of $G^*$ just described, the 
boundaries of the vertices of $G$ and $G^*$ intersect, and therefore 
the marking arrows on $\ar{G}$ induce marking arrows on $G^*$. The 
dual $\ar{G}^*$ of an arrow-marked ribbon graph $\ar{G}$ is the 
dual of its underlying ribbon graph equipped with the induced marking 
arrows.

\subsection{Partial duals and the ribbon graphs of link diagrams}
Partial duality, introduced by Chmutov in \cite{Ch1}, is an extension 
of the concept of  the dual of a cellularly embedded graph. Loosely 
speaking, a partial dual of a graph is obtained by forming the dual 
of a graph only at a given subset of edges. Partial duality has found 
a number of applications in graph theory, physics and knot theory 
(for example see \cite{Ch1, EMM, KRVT09, Mo5, VT10}). Here we are 
interested in the motivating application of partial duality: just as 
geometric duality related the two Tait graphs of a link diagram, 
partial duality relates the ribbon graphs of a link diagram. In 
this section we describe partial duality and give an overview of its 
application to link diagrams.

We will use the construction of a partial dual from \cite{Mo4}. The 
idea behind this construction is that given a set $A$ of edges the 
partial dual with respect to $A$ can be formed by `hiding' the edges 
not in $A$ and replacing them with marking arrows, forming the dual 
of the resulting arrow marked ribbon graph, and then revealing the 
hidden edges. We refer the reader to \cite{Ch1} for the original 
definition of a partial dual and for further details and examples of 
partial duals.

\begin{definition}
Let $G$ be a ribbon graph, $A\subseteq E(G)$ and $ A^c=E(G)-A$. Then the 
{\em partial dual}, $G^{A}$, of $G$ with respect to $A$ is constructed in the following way.
\begin{enumerate}
\item Present $G$ as the arrow-marked ribbon graph $\ar{G-A^c}$.
\item Take the dual of $G-A^c$. The marking arrows on $\ar{G-A^c}$ induce marking arrows on $\left(G-A^c\right)^*$.
\item $G^{A}$ is the ribbon graph corresponding to the arrow-marked ribbon graph $\ar{(G-A^c)^*}$.
\end{enumerate}
\end{definition}

We will use the convention that $A^c:=E(G)-A$ throughout this paper.

\begin{example}\label{ex.dex3} This example illustrates  the construction of a partial dual.
~

\begin{center}
\begin{tabular}{ccc}
 \includegraphics[height=2.0cm]{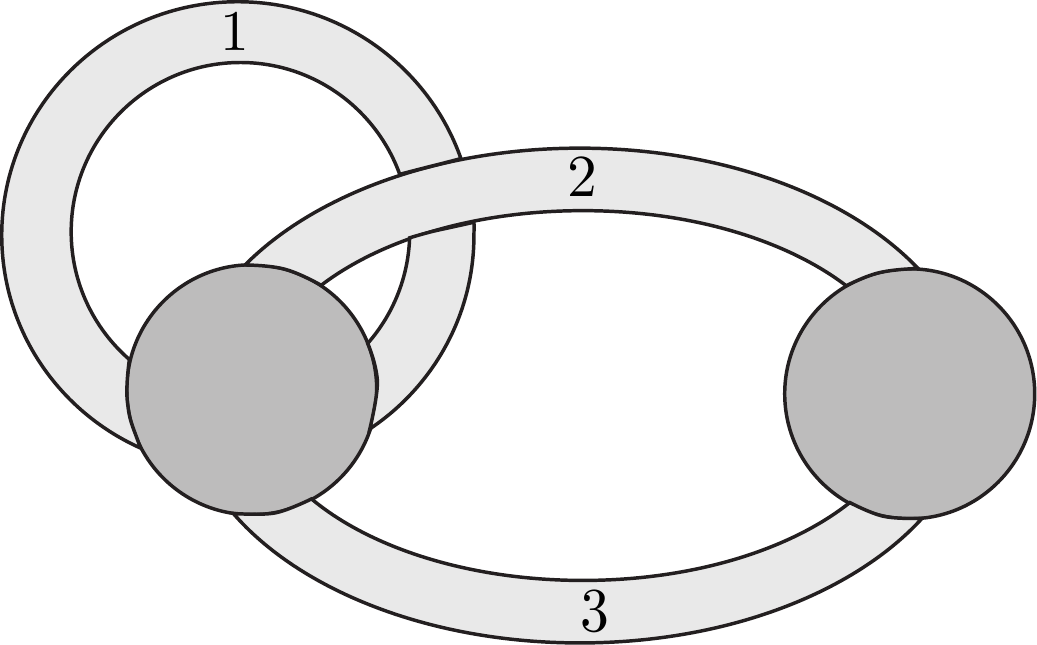}  & \hspace{5mm}\raisebox{7mm}{=}\hspace{5mm} & \includegraphics[height=2.0cm]{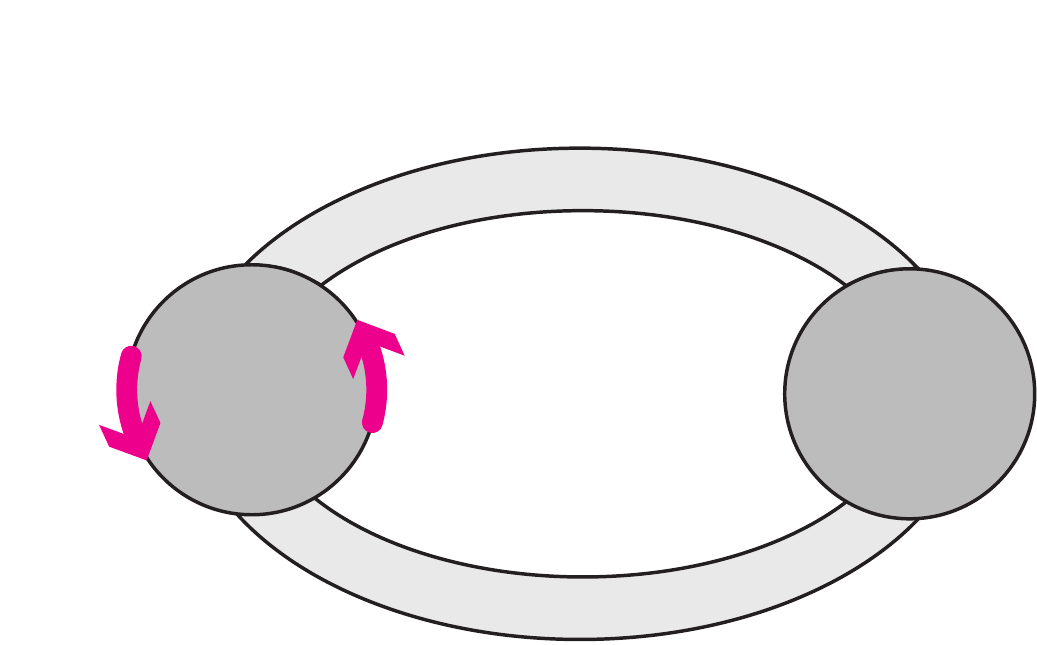}   \\ & & \\
 $G$ with $A=\{2,3\}.$  & &Step 1.  \\ 
 \end{tabular}
 \end{center}
 \begin{center}
\begin{tabular}{ccc}
  \includegraphics[height=3cm]{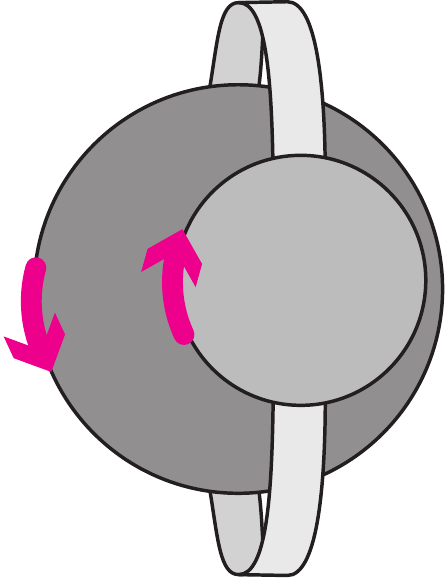}&\hspace{5mm}\raisebox{15mm}{=}\hspace{5mm}& \includegraphics[height=3cm]{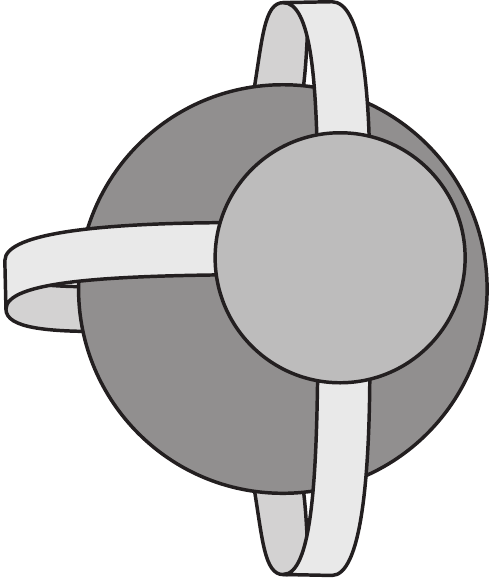}    \hspace{5mm}\raisebox{15mm}{=}\hspace{5mm} \raisebox{6mm}{\includegraphics[height=2.0cm]{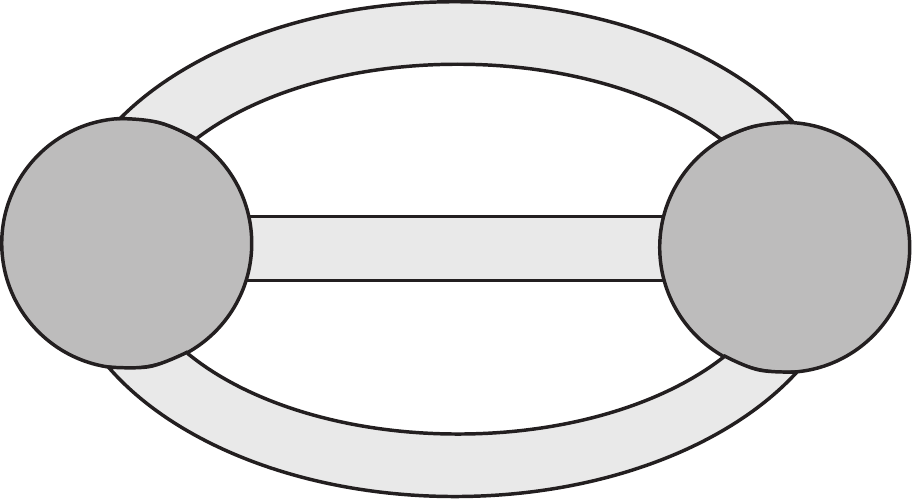}  }    \\ & & \\
Step 2. &&  $G^{A} $.
\end{tabular}
\end{center}
\end{example}

There is a natural bijection between the edges of $G$ and of $G^A$. 
We will usually use this bijection to identify the edges of $G$ with 
the edges of its partial dual.

We will need the following basic properties of partial duality. The 
first five properties are from~\cite{Ch1} and the sixth is from 
\cite{Mo5}.
\begin{proposition}\label{p.pd2}
Let $G$ be a ribbon graph and $A, B\subseteq E(G)$.  Then 
\begin{enumerate}
\item $G^{\emptyset}=G$;
\item  $G^{E(G)}=G^*$, where $G^*$ is the dual of $G$;
\item $(G^A)^B=G^{A\Delta B}$, where $A\Delta B := (A\cup B)\backslash (A\cap B)$ is the symmetric difference of $A$ and $B$;
\item $G$ is orientable if and only if $G^A$ is orientable;
\item partial duality acts disjointly on connected components;
\item If $G$ is orientable, then $g(G^{A})=\frac{1}{2}\left(2k(G)+e(G)-p(G-A^c)-p(G-A)\right)$.
\end{enumerate}
\end{proposition}



Recall that the two (bi-weighted) Tait graphs of a link diagram are 
related by duality. Analogously, partial duality relates all of the 
ribbon graphs of  a link diagram. In fact, the following result holds.
\begin{proposition}\label{p.taitdual}
A ribbon graph represents 
a link diagram if and only if it is a partial dual of a plane graph.
\end{proposition}

Partial duality provides a natural construction for the various 
ribbon graphs associated with a link diagram described in 
Section~\ref{ss.rgld}. Let $G$ be an edge weighted ribbon graph and 
$A\subseteq E(G)$. Then $G^A$ is also an edge weighted ribbon graph. 
The edge-weights of $G^A$ are determined as follows:
\begin{itemize}
 \item if an edge $e$ of $G$ has weight $m_e\in \{+,-\}$, then the 
 corresponding edge of $G^A$ has weight $-m_e$ if $e\in A$, and has 
 weight $m_e$ if $e\notin A$.
 \item if an edge $e$ of $G$ has weight $(m_e, \sigma_e) \in 
 \{+,-\}\times \{+,-\}$, then the corresponding edge of $G^A$ has 
 weight $(-m_e, \sigma_e)$ if $e\in A$, and has weight 
 $(m_e,\sigma_e)$ if $e\notin A$.
\end{itemize}

With this action of partial duality on the edge-weights we have the 
following proposition. 

\begin{proposition}\label{p.pdt} Let $D$ be an oriented link diagram. Then
\begin{enumerate}
\item all of the ribbon graphs of the link are partial duals of either 
of the Tait graphs $\T(D)$;
\item $\A(D)=\T(D)^A$, where $A$ is the set of $+$ weighted edges 
of $\T(D)$;
\item $\B(D)=\T(D)^A$, where $A$ is the set of $-$ weighted edges 
of $\T(D)$;
\item $\s(D)=\T_{\sigma}(D)^A$, where $A$ is the set of all $(+,+)$ 
and $(-,-)$ weighted edges of  $\T_{\sigma}(D)$.
\end{enumerate}
\end{proposition}
A proof of the first three statements can be found in \cite{Ch1} or \cite{Mo5}, and 
the proof of the fourth is similar and is therefore omitted.


\section{Seifert Graphs}\label{s.ab}

In this section we focus on the structure of Seifert graphs. It is 
well known that for any link diagram $D$, its Seifert ribbon graph 
$\s(D)$ is bipartite (see \cite{Cr04}, for example). Although 
biparticity is a necessary condition for a graph to be the Seifert 
graph of a link diagram, it is easily seen that it is not sufficient. 
Here we provide a necessary and sufficient condition for a graph to 
be the Seifert graph of a link diagram. A plane graph is bipartite if 
and only if its dual is Eulerian (see \cite{Bondy}, for example), and 
 our characterization of Seifert graphs will be stated in 
terms of the dual concept of Eulerian graphs.

Being bipartite or Eulerian is a property of abstract graphs rather 
than embedded graphs (in the sense that the properties are 
independent of how the graph is drawn in a surface). Accordingly, we 
will need to work with the abstract graphs of a link diagram. We will 
say that two embedded graphs $G$ and $H$ are  {\em equivalent as 
abstract graphs}, written $G\cong H$, if $G$ and $H$ are drawings of 
the same (abstract) graph. An abstract graph $G$ is a {\em graph of a 
link diagram} $D$ if $G\cong \G$, for some $\G \in \mathcal{G}(D)$. 
We note that, here, the graph of a link diagram is not signed. 
Finally, we say that an (abstract) graph $\bS(D)$ is {\em Seifert 
graph} of a link diagram $D$ if $\bS(D) \cong \s(D)$.

We will prove the following characterization of Seifert graphs.
\begin{theorem}\label{t.sc}
A graph $G$ is the Seifert graph of a link diagram $D$ if and only if 
\[G\cong [(H\cup H^*) - (A^c \cup A^*)]^{\du} =:\Phi^{\du},\]  
where $H$ is  plane, $H\cup H^*$ has the standard immersion, and each component of $\Phi$ is Eulerian.
Moreover, $A$ is the set of $(\pm, \pm)$-weighted edges in the Tait graph $\T_{\sigma}(D)$.
\end{theorem}

This theorem will follow from a careful analysis of how the Seifert 
graph is formed. This analysis revolves around an algorithm for  
obtaining the Seifert graph from a Tait graph which does not require 
the use of partial duals, {\em i.e.} we provide a way to recover 
$\bS(D)$ from $\T(D)$ without having to construct the ribbon graph 
$\s(D)$.

To obtain this algorithm, we follow the steps of the Seifert algorithm on the
link diagram $D$, while examining how each step of this algorithm
affects the graph $\T:=\T(D)$. Since Seifert graphs are not signed, we will ignore the edge weights of Tait graphs throughout this section.  We begin by resolving the crossings
of our oriented link diagram $D$, by following the orientation. Each
crossing of $D$ corresponds to one edge of the graph $\T$, and
when resolving the given crossing of our link diagram we pay careful
attention to what happens to the corresponding edge of $\T$.

At each crossing of $D$ there are two regions corresponding to
vertices of $\T$, which we will refer to as the black regions.
We can see in Figure \ref{c and d edges} that, depending on the
orientation of the components of the link $D$, when we resolve a
crossing of $D$ we either separate the black regions, or we merge
them together. If the regions are separated, the edge corresponding
to that crossing is deleted, while if the regions are merged together
the edge corresponding to that crossing is contracted. We therefore
label the edges of $\T$ with the letters $c$ for contraction and
$d$ for deletion.

It is these $c$ and $d$ labels that carry the information about the
link orientation. If we change the orientation of all the components
of $D$ then the labels will not change. If, however, we only change
the orientation of one of the components, then at a given crossing
involving that component (other than a self-crossing) the label of
the corresponding edge of $\T$ will change from $c$ to $d$ or vice 
versa.

If $D$ has $k$ components, there are $2^{k-1}$ possible $\{c,d\}$ colourings. Neither $\T$ 
nor $\bS:=\bS(D)$ carry the under- and over-crossing information, which is 
completely independent of the $\{c,d\}$ colouring. Many different 
links will therefore correspond to the same $\T$, and even more 
to the same $\bS$. $\T$ and $\bS$ have the same edges, and so a 
$\pm$ colouring of either will do for both.

\setlength{\unitlength}{1pt}
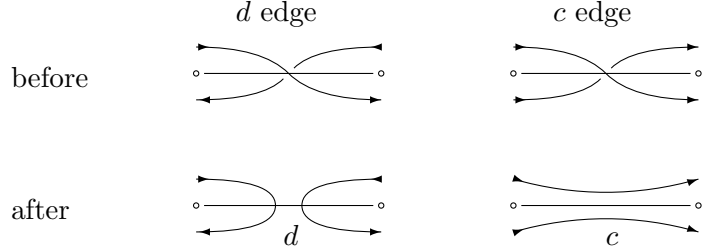
\begin{figure}
\begin{picture}(150,140)(0,0)

\put(10,55){\circle{2}} \put(80,55){\circle{2}}
\put(13,55){\line(1,0){64}} \qbezier(10,65)(35,65)(45,55)
\qbezier(45,55)(55,45)(80,45) \qbezier(10,45)(35,45)(43,53)
\qbezier(47,57)(55,65)(80,65) \put(15,65){\vector(1,0){0}}
\put(11,45){\vector(-1,0){0}} \put(76,65){\vector(-1,0){0}}
\put(80,45){\vector(1,0){0}}

\put(130,55){\circle{2}} \put(200,55){\circle{2}}
\put(133,55){\line(1,0){64}} \qbezier(130,65)(155,65)(165,55)
\qbezier(165,55)(175,45)(200,45) \qbezier(130,45)(155,45)(163,53)
\qbezier(167,57)(175,65)(200,65) \put(135,65){\vector(1,0){0}}
\put(135,45){\vector(1,0){0}} \put(200,65){\vector(1,0){0}}
\put(200,45){\vector(1,0){0}}

\put(10,5){\circle{2}} \put(80,5){\circle{2}}
\put(13,5){\line(1,0){64}} \qbezier(10,15)(40,15)(40,5)
\qbezier(50,5)(50,-5)(80,-5) \qbezier(10,-5)(40,-5)(40,5)
\qbezier(50,5)(50,15)(80,15) \put(43,-10){$d$}
\put(15,15){\vector(1,0){0}} \put(11,-5){\vector(-1,0){0}}
\put(76,15){\vector(-1,0){0}} \put(80,-5){\vector(1,0){0}}

\put(130,5){\circle{2}} \put(200,5){\circle{2}}
\put(133,5){\line(1,0){64}} \qbezier(130,15)(165,5)(200,15)
\qbezier(130,-5)(165,5)(200,-5) \put(165,-10){$c$}
\put(134,14){\vector(3,-1){0}} \put(134,-4){\vector(3,1){0}}
\put(200,15){\vector(3,1){0}} \put(200,-5){\vector(3,-1){0}}

\put(25,75){$\textrm{$d$ edge}$} \put(145,75){$\textrm{$c$ edge}$}
\put(-60,50){$\textrm{before}$} \put(-60,0){$\textrm{after}$}

\end{picture}
\caption{$c$ and $d$ edges before and after resolution of a
crossing}
\label{c and d edges}
\end{figure}

Our next task is to characterize those $\{c,d\}$ labellings of the
edges of an arbitrary plane graph which come from link diagrams using the above process.

Let $E_{c}$ be the set of edges of $\T$ coloured $c,$ and let
$V_{c}$ be the set of vertices of $\T$ adjacent to a $c$ edge.
Let $C=(V_{c},E_{c})$. The graph is not necessarily connected.

The $d$
edges in $\T$ correspond to $c$ edges in $\T^{*}$. Let
$E'_{c}$ be the set of edges of $\T^{*}$ coloured $c,$ and let
$V'_{c}$ be the set of vertices of $\T^{*}$ adjacent to a $c$
edge. Then as before we put $C'=(V'_{c},E'_{c})$, and again $C'$ is not 
necessarily connected.

\begin{theorem}\label{Even c degree}
Each vertex of graph $\T$ has an even number of $c$ edges
adjacent to it (with loops being counted twice).
\end{theorem}

\begin{proof}

Consider a vertex $v$ of $\T$, with edges
$e_{1},e_{2},\ldots,e_{n}$ adjacent to it. This vertex $v$
corresponds to a face $f$ of $D$, and this face $f$ is bounded
by arcs of the link diagram with crossings $c_{1},c_{2},\ldots,c_{n}$
corresponding to the edges $e_{1},e_{2},\ldots,e_{n}$ of $\T$.
Let us take a position on one of those arcs of $D$, between the edges
$e_{n}$ and $e_{1}$ of $\T$, and then walk around the face $f$
until we return to our starting point. We note whether the direction
we walk in agrees or disagrees with the orientation of $D$.

Without loss of generality we may assume that we begin walking in
the direction compatible with the orientation of the link diagram
$D$. We walk until we meet our first edge $e_{1}$. This edge has a
label, either $c$ or $d$. Let us assume that the edge $e_{1}$ is a
$c$ edge; see Figure \ref{walk}. When we cross $e_{1}$ and continue
walking onto the next strand we will walk against the direction of
the orientation of $D$. In order for us to get back where we started
and walk in the direction of the orientation of $D$, the arcs of $D$
will have to change the direction in total an even number of times.
Therefore there must be an even number of $c$ edges at each vertex
$v$ of $\T$.

Now let us assume that the first edge we met, $e_{1}$, is a $d$ edge.
When we cross this edge the next arc of $D$ has to be oriented in a
direction compatible with the direction we are walking in, otherwise
the edge $e_{1}$ would be a $c$ edge. So we keep walking until we
meet a $c$ edge, and then the previous argument applies. If we do not
encounter any $c$ edges then all the edges adjacent to the vertex $v$
are labelled $d$. \end{proof}

\begin{figure}
\begin{picture}(100,120)(0,0)
\thicklines \put(50,50){\circle{2}} \put(43,43){$v$}
\put(53,50){\line(1,0){70}} \put(125,47){$\textrm{$e_{n}$}$}
\put(52,52){\line(1,2){30}} \put(85,115){$\textrm{$e_{1}$}$}
\put(48,52){\line(-1,2){30}} \put(15,115){$\textrm{$e_{2}$}$}
\put(52,48){\line(1,-2){30}} \put(85,-10){$\textrm{$e_{n-1}$}$}
\put(47,50){\line(-1,0){70}} \put(-33,47){$\textrm{$e_{3}$}$}
\put(60,60){$c$} \put(80,95){$\textrm{$c_{1}$}$}
\put(33,60){$d$} \put(12,98){$\textrm{$c_{2}$}$}

\thinlines \qbezier(128,40)(75,65)(75,120)
\qbezier(95,110)(50,85)(5,110) \qbezier(25,120)(25,65)(-30,40)

\put(88,73){\vector(-1,2){0}} \put(75,120){\vector(0,1){0}}
\put(98,112){\vector(3,2){0}} \put(52,97){\vector(1,0){0}}
\put(25,120){\vector(0,1){0}} \put(-28,41){\vector(3,2){0}}
\put(9,108){\vector(3,-2){0}} \put(11,72){\vector(1,2){0}}

\end{picture}
\caption{Walking around vertex $v$ of graph $\T$}
\label{walk}
\end{figure}
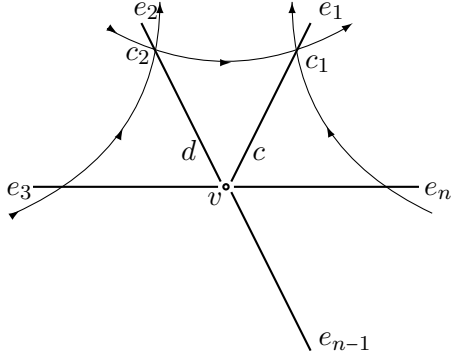

\begin{corollary}\label{C subgraph Eulerian}
Each connected component of $C$ is Eulerian.
\end{corollary}

\begin{proof}
This comes from the following equivalences: a graph $G$ is Eulerian 
if an only if every vertex of $G$ has an even degree, and if and only 
if the set of edges of $G$ can be partitioned into cycles. (For a 
proof of these equivalences see \cite{Bondy}, for example.)
\end{proof}

\begin{corollary}\label{C' subgraph Eulerian}
Each connected component of $C'$ is Eulerian.
\end{corollary}

\begin{proof}
This follows since $C'$ is the Tait graph obtained from $D$ using the 
other checkerboard colouring.

\end{proof}

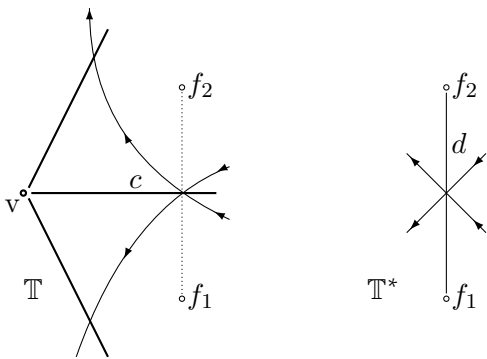
\begin{figure}
\begin{picture}(100,150)(60,0)

\thicklines \put(50,50){\circle{2}} \put(43,43){$\textrm{v}$}
\put(53,50){\line(1,0){70}} \put(90,52){$c$}
\put(52,52){\line(1,2){30}} \put(52,48){\line(1,-2){30}}

\thinlines \qbezier(128,40)(75,65)(75,120)
\qbezier(128,60)(90,45)(70,-12)

\put(88,73){\vector(-1,2){0}} \put(75,120){\vector(0,1){0}}
\put(122,43){\vector(-2,1){0}} \put(123,58){\vector(-2,-1){0}}
\put(88,25){\vector(-1,-2){0}}

\put(110,10){\circle{2}} \put(110,90){\circle{2}}
\put(112,8){$f_{1}$} \put(112,88){$f_{2}$}
\qbezier[50](110,12)(110,50)(110,88)

\put(210,10){\circle{2}} \put(210,90){\circle{2}}
\put(212,8){$f_{1}$} \put(212,88){$f_{2}$}
\qbezier(210,12)(210,50)(210,88) \put(195,35){\line(1,1){30}}
\put(225,35){\line(-1,1){30}} \put(196,36){\vector(-1,-1){0}}
\put(220,40){\vector(-1,1){0}} \put(196,64){\vector(-1,1){0}}
\put(220,60){\vector(-1,-1){0}} \put(212,65){$d$}

\put(50,10){$ \T $} \put(180,10){$
\T^{\ast} $}

\end{picture}
\caption{Relationship between the labels of edges of $\T$
and $\T^{\ast}$.} \label{swap}
\end{figure}

\begin{corollary}\label{Even d cycles}
Any cycle of $\T$ contains an even number of $d$ edges.
\end{corollary}

\begin{proof}

From Corollary \ref{C subgraph Eulerian} we can see that $C$ can be 
partitioned into a set of cycles. $\T$ is a
plane graph so each cycle of $\T$ separates the sphere into two
regions (by the Jordan curve theorem). If either of these two regions
contains no other edges of $\T$ then we call the cycle a {\it
boundary}.

First we show that a boundary cycle of $\T$ contains an even number
of $d$ edges. Let $C_{1}$ be a boundary cycle of $\T$. We also have 
the dual Tait graph $\T^{*}$ corresponding to the same
link diagram. The cycle $C_{1}$ in $\T$ corresponds to a star
centered at vertex $v^{*}$ in $\T^{*}$. Each edge $e$ of $C_{1}$
has a corresponding edge $e^{*}$ adjacent to $v^{*}$, with the
opposite $\{c,d\}$ label; see Figure \ref{swap}. Applying Theorem
\ref{Even c degree} to $\T^{*}$ we deduce that each vertex $v^{*}$
has an even number of $c$ edges adjacent to it. Hence the cycle $C_{1}$
of $\T$ contains an even number of $d$ edges.

Now let $C_{2}$ be another boundary cycle, whose intersection with 
$C_{1}$ is a path. However many $d$ edges there are in this path, if 
we were to delete the whole path the new face would have an even 
number of $d$ edges. So a (non-boundary) cycle consisting of two 
faces must have an even number of $d$ edges. Inductively, therefore, 
we see that any cycle contains an even number of $d$ edges.

\end{proof}

Note that, in fact, Theorem~\ref{Even c degree} and Corollaries~\ref{C
subgraph Eulerian}, \ref{C' subgraph Eulerian}, and \ref{Even d cycles}
are all equivalent to each other.

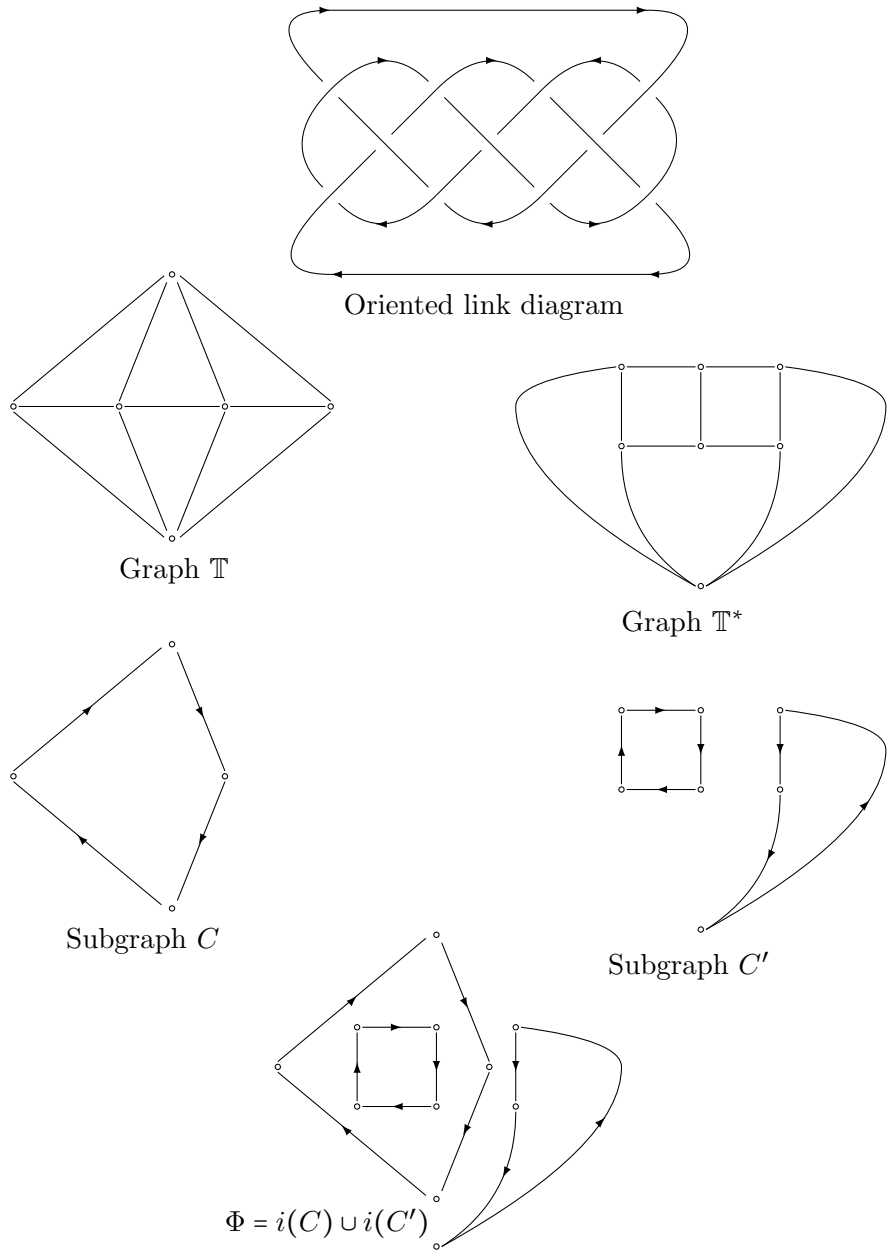
\begin{figure}
\begin{picture}(320,470)(0,0)

\put(100,400){\begin{picture}(300,110)(0,0)
\qbezier(30,80)(10,60)(27,43) \qbezier(33,77)(50,60)(67,43)
\qbezier(30,40)(30,40)(47,57) \qbezier(53,63)(70,80)(70,80)
\qbezier(30,80)(50,100)(67,83) \qbezier(33,37)(50,20)(70,40)
\qbezier(73,77)(90,60)(107,43) \qbezier(70,40)(70,40)(87,57)
\qbezier(93,63)(110,80)(110,80) \qbezier(70,80)(90,100)(107,83)
\qbezier(73,37)(90,20)(110,40) \qbezier(113,77)(130,60)(147,43)
\qbezier(110,40)(110,40)(127,57) \qbezier(133,63)(150,80)(150,80)
\qbezier(110,80)(130,100)(147,83) \qbezier(113,37)(130,20)(150,40)
\qbezier(153,77)(170,60)(150,40) \qbezier(27,83)(0,110)(30,110)
\qbezier(30,110)(90,110)(150,110)
\qbezier(150,110)(180,110)(150,80) \qbezier(30,40)(0,10)(30,10)
\qbezier(30,10)(90,10)(150,10) \qbezier(150,10)(180,10)(153,37)

\put(30,110){\vector(1,0){0}} \put(150,110){\vector(1,0){0}}
\put(87,29){\vector(-1,0){0}} \put(30,10){\vector(-1,0){0}}
\put(150,10){\vector(-1,0){0}} \put(93,91){\vector(1,0){0}}
\put(52,91){\vector(1,0){0}} \put(128,91){\vector(-1,0){0}}
\put(47,29){\vector(-1,0){0}} \put(132,29){\vector(1,0){0}}

\put(35,-5){$\textrm{Oriented link diagram}$}
\end{picture}}

\put(0,300){\begin{picture}(130,90)(0,0) \put(10,60){\circle{2}}
\put(50,60){\circle{2}}
\put(90,60){\circle{2}}\put(130,60){\circle{2}}
\put(70,10){\circle{2}} \put(70,110){\circle{2}}
\put(12,60){\line(1,0){36}} \put(52,60){\line(1,0){36}}
\put(92,60){\line(1,0){36}} \put(50,62){\line(2,5){18}}
\put(50,58){\line(2,-5){18}} \put(130,62){\line(-6,5){57}}
\put(130,58){\line(-6,-5){57}} \put(10,62){\line(6,5){57}}
\put(10,58){\line(6,-5){57}} \put(90,62){\line(-2,5){18}}
\put(90,58){\line(-2,-5){18}}

\put(50,-5){$\textrm{Graph $\T$}$} \end{picture}}

\put(200,300){\begin{picture}(130,90)(0,0)

\put(40,45){\circle{2}} \put(40,75){\circle{2}}
\put(70,45){\circle{2}}\put(70,75){\circle{2}}
\put(100,45){\circle{2}} \put(100,75){\circle{2}}
\put(70,-8){\circle{2}}

\put(40,47){\line(0,1){26}} \put(70,47){\line(0,1){26}}
\put(100,47){\line(0,1){26}} \put(42,75){\line(1,0){26}}
\put(72,75){\line(1,0){26}} \put(42,45){\line(1,0){26}}
\put(72,45){\line(1,0){26}} \qbezier(40,43)(40,10)(68,-8)
\qbezier(100,43)(100,10)(72,-8) \qbezier(38,75)(0,70)(0,60)
\qbezier(0,60)(0,30)(68,-8) \qbezier(102,75)(140,70)(140,60)
\qbezier(140,60)(140,30)(72,-8)

\put(40,-25){$\textrm{Graph $\T^{*}$}$} \end{picture}}

\put(0,160){\begin{picture}(130,90)(0,0) \put(10,60){\circle{2}}
\put(90,60){\circle{2}} \put(70,10){\circle{2}}
\put(70,110){\circle{2}} \put(10,62){\line(6,5){56}}
\put(10,58){\line(6,-5){56}} \put(90,62){\line(-2,5){18}}
\put(90,58){\line(-2,-5){18}} \put(40,87){\vector(1,1){0}}
\put(34,38){\vector(-1,1){0}} \put(82,82){\vector(1,-2){0}}
\put(80,34){\vector(-1,-2){0}}

\put(30,-5){$\textrm{Subgraph $C$}$}
\end{picture}}

\put(200,170){\begin{picture}(130,90)(0,0) \put(40,45){\circle{2}}
\put(40,75){\circle{2}}
\put(70,45){\circle{2}}\put(70,75){\circle{2}}
\put(100,45){\circle{2}} \put(100,75){\circle{2}}
\put(70,-8){\circle{2}} \put(40,47){\line(0,1){26}}
\put(42,75){\line(1,0){26}} \put(70,47){\line(0,1){26}}
\put(42,45){\line(1,0){26}} \put(100,47){\line(0,1){26}}
\qbezier(100,43)(100,10)(72,-8) \qbezier(102,75)(140,70)(140,60)
\qbezier(140,60)(140,30)(72,-8) \put(40,62){\vector(0,1){0}}
\put(57,75){\vector(1,0){0}} \put(70,58){\vector(0,-1){0}}
\put(53,45){\vector(-1,0){0}} \put(100,58){\vector(0,-1){0}}
\put(95,18){\vector(-1,-2){0}} \put(134,41){\vector(1,1){0}}

\put(35,-25){$\textrm{Subgraph $C'$}$}
\end{picture}}

\put(100,50){\begin{picture}(130,90)(0,0)

\put(10,60){\circle{2}} \put(90,60){\circle{2}}
\put(70,10){\circle{2}} \put(70,110){\circle{2}}
\put(10,62){\line(6,5){56}} \put(10,58){\line(6,-5){56}}
\put(90,62){\line(-2,5){18}} \put(90,58){\line(-2,-5){18}}
\put(40,87){\vector(1,1){0}} \put(34,38){\vector(-1,1){0}}
\put(82,82){\vector(1,-2){0}} \put(80,34){\vector(-1,-2){0}}

\put(40,45){\circle{2}} \put(40,75){\circle{2}}
\put(70,45){\circle{2}}\put(70,75){\circle{2}}
\put(100,45){\circle{2}} \put(100,75){\circle{2}}
\put(70,-8){\circle{2}} \put(40,47){\line(0,1){26}}
\put(42,75){\line(1,0){26}} \put(70,47){\line(0,1){26}}
\put(42,45){\line(1,0){26}} \put(100,47){\line(0,1){26}}
\qbezier(100,43)(100,10)(72,-8) \qbezier(102,75)(140,70)(140,60)
\qbezier(140,60)(140,30)(72,-8) \put(40,62){\vector(0,1){0}}
\put(57,75){\vector(1,0){0}} \put(70,58){\vector(0,-1){0}}
\put(53,45){\vector(-1,0){0}} \put(100,58){\vector(0,-1){0}}
\put(95,18){\vector(-1,-2){0}} \put(134,41){\vector(1,1){0}}

\put(-10,-2){$\textrm{$\Phi = i(C) \cup i(C')$}$}
\end{picture}}

\end{picture}
\caption{Construction of the graph $\Phi$}
\label{dupa}
\end{figure}

The underlying abstract graph $\boldsymbol{T}$ of $\T$  is equipped with an
embedding $i:\boldsymbol{T} \rightarrow S^{2}$. This induces an embedding
$\boldsymbol{T}^{*}\rightarrow S^{2}$, which we also denote $i$. Now take
$i(C)\cup i(C')$, and denote the resulting graph
by $\Phi$. 
So $\Phi$ is formed by taking the standard immersion of $\T\cup\T^*$ and then setting
\[ \Phi:= (\T\cup\T^*)  -  (E(C^c) \cup E((C')^c)).  \]
The construction of $\Phi$ and $\Phi^{\du}$ can be seen in
Figures \ref{dupa} and \ref{Phi*}. Note that $\Phi^{\du}$ is not
always a thickened tree. The graph $\Phi^{\du}$ will turn out to be the 
Seifert graph $\bS$.

\begin{figure}
\begin{picture}(320,170)(0,0)

\put(20,20){\begin{picture}(130,90)(0,0)

\put(10,60){\circle{2}} \put(90,60){\circle{2}}
\put(70,10){\circle{2}} \put(70,110){\circle{2}}
\put(10,62){\line(6,5){56}} \put(10,58){\line(6,-5){56}}
\put(90,62){\line(-2,5){18}} \put(90,58){\line(-2,-5){18}}


\put(40,45){\circle{2}} \put(40,75){\circle{2}}
\put(70,45){\circle{2}}\put(70,75){\circle{2}}
\put(100,45){\circle{2}} \put(100,75){\circle{2}}
\put(70,-8){\circle{2}} \put(40,47){\line(0,1){26}}
\put(42,75){\line(1,0){26}} \put(70,47){\line(0,1){26}}
\put(42,45){\line(1,0){26}} \put(100,47){\line(0,1){26}}
\qbezier(100,43)(100,10)(72,-8) \qbezier(102,75)(140,70)(140,60)
\qbezier(140,60)(140,30)(72,-8)


\put(52,57){$\textrm{1}$} \put(75,57){$\textrm{2}$}
\put(92,87){$\textrm{3}$} \put(112,57){$\textrm{4}$}

\put(-5,-5){$\textrm{Graph $ \Phi $}$}
\end{picture}}

\put(220,20){\begin{picture}(130,90)(0,0) \put(0,60){\circle{2}}
\put(40,60){\circle{2}} \put(80,60){\circle{2}}
\put(120,60){\circle{2}}

\put(3,63){\line(1,0){34}} \put(3,61){\line(1,0){34}}
\put(3,59){\line(1,0){34}} \put(3,57){\line(1,0){34}}
\put(43,63){\line(1,0){34}} \put(43,61){\line(1,0){34}}
\put(43,59){\line(1,0){34}} \put(43,57){\line(1,0){34}}
\put(83,60){\line(1,0){34}} \put(83,62){\line(1,0){34}}
\put(83,58){\line(1,0){34}}

\put(-3,65){$\textrm{1}$} \put(37,65){$\textrm{2}$}
\put(77,65){$\textrm{3}$} \put(117,65){$\textrm{4}$}

\put(35,25){$\textrm{Graph $\Phi^{\du}$}$}
\end{picture}}

\end{picture}
\caption{Construction of graph $\Phi^{\du}$}
\label{Phi*}
\end{figure}

\begin{theorem}
\label{phi*Seifert}
$\Phi^{\du}$ is the Seifert graph of the link diagram $D$
corresponding to $\T$.
\end{theorem}

\begin{proof}
We construct a mapping $h$ from the vertex set $V(\bS)$ of the
Seifert graph to the vertex set $V(\Phi^{\du})$. We draw on the
same diagram the link diagram $D$ together with the corresponding
graphs $\T$ and $\T^{*}$, so that for each crossing of $D$
there are two corresponding edges, one from $\T$ and one from
$\T^{*}$. While resolving the crossings of the link diagram we
remove the $d$ edges of $\T$ and $\T ^{*}$, so we have one
edge for each crossing of $D$. Now we observe that the Seifert
circles do not intersect any of the edges, and therefore given a
Seifert circle we have a face of $\Phi$, and hence a vertex of
$\Phi^{\du}$ as required.

Suppose there are two Seifert circles in the same face of $\Phi$. Each
edge of this face crosses one of the dotted edges emerging from the
Seifert circles, and we can label the edges $1$ or $2$ according to
which Seifert circle they correspond to. There are just two vertices
$a$ and $b$ in the boundary of the face which have both $1$ and $2$
edges. These vertices are in the same face of $U$, so $a=b$, and hence
the two circles must have been in different faces of $\Phi$.

So $h$ is injective.

We also know that if two vertices $v$ and $w$ in $\bS$ are adjacent,
then the corresponding Seifert circles are joined by a dotted edge,
which implies that the faces in $\Phi$ share an edge and hence that
$h(v)$ and $h(w)$ are adjacent in $\Phi^{\du}$.

This gives us a mapping $m:E(\bS)\rightarrow E(\Phi^{\du})$, which must
be a bijection, because any edge of $\Phi^{\du}$ comes ultimately from
a specific crossing in $D$ and hence from a specific edge of $\bS$.

Finally, we show that $h$ is a graph isomorphism. We already know that
it is injective. Now let $\lambda$ be in $V(\Phi^{\du})$. Choose a vertex
$\mu$ adjacent to $\lambda$, denote $m^{-1}(\lambda\mu)$ by $e$, and
suppose that $e=uv$, for $u,v\in V(\bS)$. Then $h(u)h(v)=\lambda\mu\in
E(\Phi^{\du})$, which means that $\lambda=h(u)$ or $\lambda=h(v)$, and
hence $h$ is surjective.
\end{proof}

It is well known that Seifert graphs are bipartite. Here we see that 
this fact is a consequence of the Eulerian structures in Tait graphs.
\begin{lemma}
$\Phi^{\du}$ is bipartite. \label{bip}
\end{lemma}

\begin{proof}
Each component of the (disjoint) graphs $C$ and $C'$ is Eulerian, and 
so each vertex of $\Phi$ has even degree. Hence each face of $\Phi^{\du}$
has an even number of edges.

Any cycle in $\Phi^{\du}$ can be formed by adding face-cycles mod 2.
Hence the result. \end{proof}

Now we establish the conditions on a connected plane graph $G$ which 
make it the Tait graph of a link diagram.

\begin{theorem}
\label{cd colourings giving links}
If, in a given $\{c,d\}$ colouring of the edges of the connected 
plane graph $G$, each component of $C$ is Eulerian, and each 
component of $C'$ is Eulerian, then $G$ is the Tait graph of an 
oriented link diagram, the orientation coming from the $\{c,d\}$ 
colouring.
\end{theorem}

\begin{proof}
Given our graph $G$ we construct $C$, $C'$, $\Phi$, and $\Phi^{\du}$ as 
above. We first show that the faces of $\Phi$ must be either discs or 
annuli, by arguing that no two components of $C'$ can lie in the same 
face of $i(C)$.

Consider a face $f$ of $i(C)$, drawn in $G$. If it is a face of $G$, 
then it is a disc. If it is not a face of $G$, then it contains other 
edges and vertices of $G$, all the edges being marked $d$. These edges 
and vertices divide $f$ into faces of $G$. But $G$ is connected, so 
we can walk between any two of these faces via other such faces. 
Therefore, the part of $C'$ lying in $f$ must be connected, as 
required.

In $\Phi^{\du}$, this means that the deletion of any cut vertex splits 
$\Phi^{\du}$ into exactly two components. We also know, from Lemma 
\ref{bip}, that $\Phi^{\du}$ is bipartite.

Choose a block $B_{1}$ of $\Phi^{\du}$. It is bipartite, and has no cut 
vertices. Choose an arbitrary transverse orientation of an edge $e^{\du}$ in 
$B_{1}$, and use this to determine a clockwise and anti-clockwise 
orientation at each of the ends of $e^{\du}$. Any other vertex in 
$B_{1}$ will be clockwise if it is an even distance from a clockwise 
vertex, and anti-clockwise otherwise. This is consistent because 
$B_{1}$ is bipartite, and so it has no odd cycles.

Next, back in $\Phi^{\du}$, let $p$ be a cut vertex of $B_{1}$. (If 
$B_{1}$ has no cut vertex, then $B_{1}=\Phi^{\du}$ and we move to the 
next step.) By our observation above, there is a next block: call it 
$B_{2}$ and orient it as in Figure~\ref{blocks}. Proceed in this way 
until all the edges of $\Phi^{\du}$ have been oriented.
Finally, transfer this orientation back to the edges of $\Phi$. 

We will now describe how an oriented link diagram $D$, with the 
property that $G$ is its Tait graph, can be recovered from the 
decorated graph $\Phi$. On the plane embedding of $G$ draw $C$ and 
$C'$ using the embedding $i$ described above. (So $C$ is a subgraph 
of $G$ and $C'$ is a subgraph of the standard embedding of $G^*$.) 
Form the link diagram $D$ from this as follows. Place a crossing at 
the centre of each  edge in $C$, and wherever an edge of $C'$ 
intersects an edge of $G$. Orient each of these crossings so that the 
arcs are directed in the same direction as the edge of $C$ or $C'$ on 
which it lies. It remains to connect these crossings up in a way that 
is constant with their orientations. First connect all crossings that 
lie on $C$ by following the faces of $C$, and all crossings that lie 
on $C'$ by following the faces of $C'$. This can be done consistently 
because of the way that we assigned transverse orientations to the 
edges in each block of $\Phi^{\du}$. This gives an oriented link 
diagram $D'$.  Finally, For each half-edge $h$ of $G$ with one end on a 
vertex $v$ of $C\subset G$ and the other on an intersection point of  
$G$ and $C'$, splice together the arcs of the link diagram $D'$ as 
indicated in Figure~\ref{f.splice}. This is consistant with the 
orientation of $D'$ because of the way that the transverse 
orientation moved between the blocks of $\Phi^{\du}$. The resulting 
link diagram is $D$.



\end{proof}

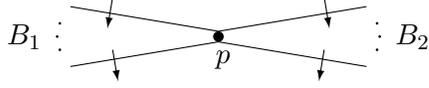
\begin{figure}
\begin{picture}(100,100)(-50,0)
    
\put(10,62){\line(-6,1){56}}
\put(10,58){\line(-6,-1){56}}

\put(-50,65){\circle*{1}}
\put(-51,60){\circle*{1}}
\put(-50,55){\circle*{1}}

\put(-70,57){$B_{1}$}

\put(10,60){\circle*{4}}
\put(9,50){$p$}

\put(10,62){\line(6,1){56}}
\put(10,58){\line(6,-1){56}}

\put(70,65){\circle*{1}}
\put(71,60){\circle*{1}}
\put(70,55){\circle*{1}}

\put(77,57){$B_{2}$}

\put(-30,75){\vector(-1,-6){2}}
\put(-30,55){\vector(1,-6){2}}

\put(50,75){\vector(1,-6){2}}
\put(50,55){\vector(-1,-6){2}}

\end{picture}

\caption{Orientation of neighbouring blocks}
\label{blocks}
\end{figure}

\begin{figure}
\begin{tabular}{ccccc}
\includegraphics[width=30mm]{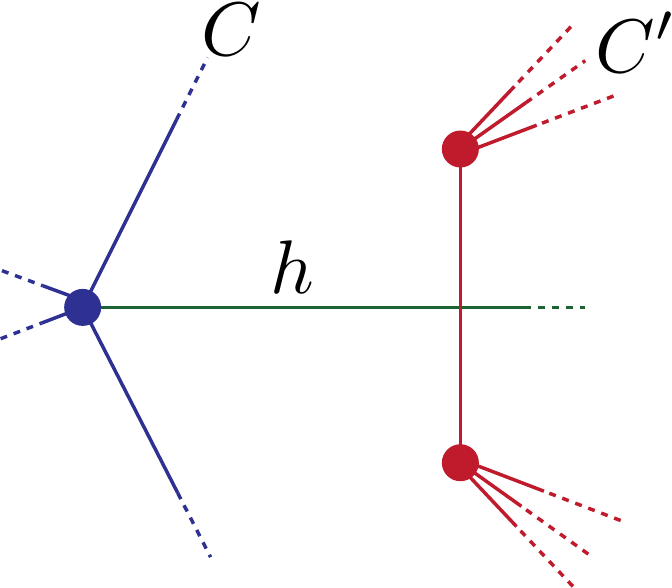} & \hspace{1cm} & 
\includegraphics[width=30mm]{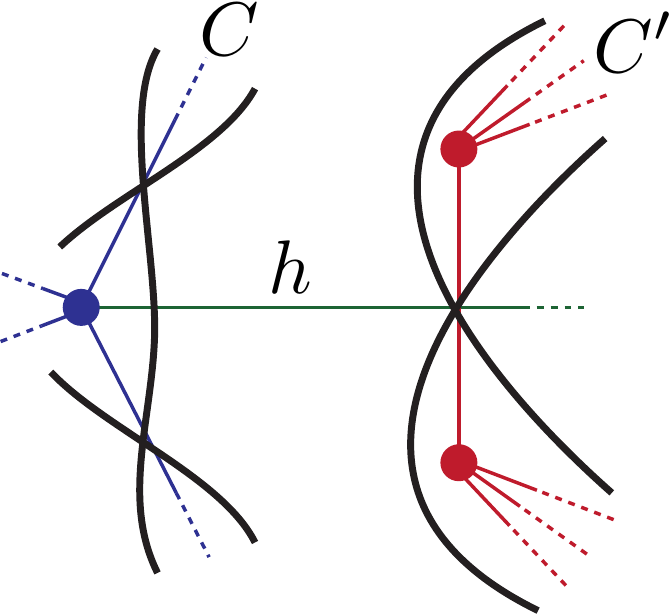}  &  \hspace{1cm} & 
\includegraphics[width=30mm]{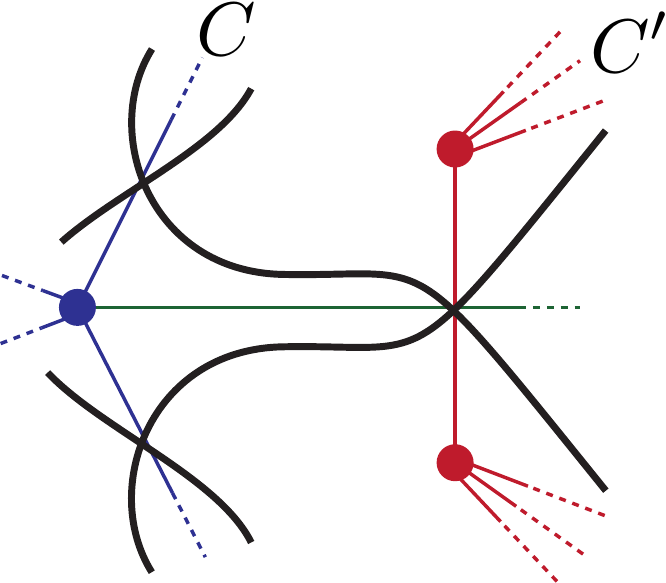} \\
In $G\cup i(C')$. && In the link $D'$. && Obtaining $D$.
\end{tabular}
\caption{Splicing together arcs in the proof of Theorem~\ref{cd 
colourings giving links}.}
\label{f.splice}
\end{figure}



















We can now prove our main result of this section, which was a 
characterization of Seifert graphs.
\begin{proof}[Proof of Theorem~\ref{t.sc}]
Suppose that $H$ is a plane graph and each component of $(H\cup H^*) 
- (A^c \cup A^*)$ is Eulerian. Then, by Theorem~\ref{cd colourings 
giving links}, it follows that $H$ is a Tait graph of a link diagram 
$D$ and that $A$ is its set of $c$ edges. By the definition of $\Phi$ 
and Theorem~\ref{phi*Seifert}, it follows that $G$ is the Seifert 
graph of $D$.

The converse follows immediately from Theorem~\ref{phi*Seifert}.
\end{proof}

\begin{remark}
Theorem~\ref{phi*Seifert} provides a way to obtain the Seifert graph 
$\boldsymbol{S}(D)$ of a link diagram $D$. By Proposition~\ref{p.pdt}, the 
Seifert graph can also be obtained as the underlying abstract graph 
of the partial dual $\T_{\sigma}(D)^A$, where $A$ is the set of all 
$(\pm,\pm)$ edges of  $\T_{\sigma}(D)$. Thus we have 
\begin{equation}\label{e.remark}
G^A \cong [(G\cup G^{\star}) - (A^c \cup A^{\star})]^{\du},
\end{equation}
when $G$ is a Tait graph and $A$ is its set of $c$ edges.
It is natural to ask whether Equation~\eqref{e.remark} holds for any 
graph $G$ and any $A\subseteq E(G)$. In \cite{HM11}, using a 
characterization of partial duals from \cite{Mo4}, the first two 
authors show that Equation~\eqref{e.remark} does indeed hold for all 
$G$ and $A$. This is significant, as it provides a way to construct 
partial dual graphs without having to compute the corresponding 
ribbon graphs. It also provides another illustration of the fruitful 
connection between knot theory and graph theory.
\end{remark}


\section{The graphs of $r$-fold parallels}
This section is concerned with the structure of the (ribbon) graphs 
of $r$-fold parallels of  link diagrams. Forming a parallel of a link 
is a fundamental and important operation.  We begin by describing the 
operation on Tait graphs that corresponds to taking the $r$-fold 
parallel of a link diagram. This construction leads to an interesting 
sequence of Tait graphs. We then apply our results on Tait graphs to 
examine the ribbon graphs of $r$-fold parallels. In particular, we 
will determine the genus of the ribbon graph $\G(D_r)$ of an 
$r$-fold parallel of $D$ in terms of the ribbon graph $\G(D)$ of $D$.

\subsection{Tait graphs of $r$-fold parallels}\label{ss.rfold}

Let $D$ be a link diagram. Its {\em $r$-fold parallel}, $D_{r}$, is 
the diagram in which each link component of $D$ is replaced by $r$ 
copies, all parallel in the plane, each copy repeating the over- and 
under-crossing behaviour of the original link diagram. 
Figure~\ref{f.t12}  shows what happens to a crossing in the 2-fold 
parallel, together with the corresponding edges in the Tait graphs. 

\setlength{\unitlength}{0.5pt}

\setlength{\unitlength}{0.5pt}
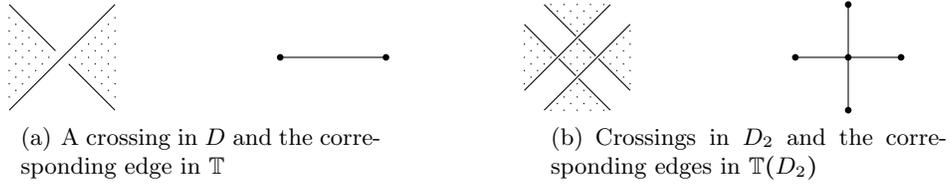
\begin{figure}
    \subfigure[A crossing in $D$ and the corresponding edge in $\T$]
    { \hspace{17mm}
      \begin{picture}(50,50)(150,0)

\multiput(110,70)(10,0){1}{\circle*{1}}
\multiput(105,65)(10,0){1}{\circle*{1}}
\multiput(100,60)(10,0){2}{\circle*{1}}
\multiput(95,55)(10,0){2}{\circle*{1}}
\multiput(90,50)(10,0){3}{\circle*{1}}
\multiput(85,45)(10,0){3}{\circle*{1}}
\multiput(80,40)(10,0){4}{\circle*{1}}
\multiput(85,35)(10,0){3}{\circle*{1}}
\multiput(90,30)(10,0){3}{\circle*{1}}
\multiput(95,25)(10,0){2}{\circle*{1}}
\multiput(100,20)(10,0){2}{\circle*{1}}
\multiput(105,15)(10,0){1}{\circle*{1}}
\multiput(110,10)(10,0){1}{\circle*{1}}

\multiput(30,70)(10,0){1}{\circle*{1}}
\multiput(35,65)(10,0){1}{\circle*{1}}
\multiput(30,60)(10,0){2}{\circle*{1}}
\multiput(35,55)(10,0){2}{\circle*{1}}
\multiput(30,50)(10,0){3}{\circle*{1}}
\multiput(35,45)(10,0){3}{\circle*{1}}
\multiput(30,40)(10,0){4}{\circle*{1}}
\multiput(35,35)(10,0){3}{\circle*{1}}
\multiput(30,30)(10,0){3}{\circle*{1}}
\multiput(35,25)(10,0){2}{\circle*{1}}
\multiput(30,20)(10,0){2}{\circle*{1}}
\multiput(35,15)(10,0){1}{\circle*{1}}
\multiput(30,10)(10,0){1}{\circle*{1}}
	
\put(30,0){\line(1,1){80}}
\put(110,0){\line(-1,1){35}}
\put(65,45){\line(-1,1){35}}

\put(235,40){\line(1,0){80}}
\put(235,40){\circle*{5}}
\put(315,40){\circle*{5}}

\end{picture} \hspace{17mm}
    }
    \hspace{2cm}
    \subfigure[Crossings in $D_{2}$ and the corresponding edges in 
    $\T(D_{2})$]
    {  \hspace{15mm}
      \begin{picture}(100,100)(150,0)

\multiput(110,60)(10,0){1}{\circle*{1}}
\multiput(105,55)(10,0){1}{\circle*{1}}
\multiput(100,50)(10,0){2}{\circle*{1}}
\multiput(95,45)(10,0){2}{\circle*{1}}
\multiput(90,40)(10,0){3}{\circle*{1}}
\multiput(95,35)(10,0){2}{\circle*{1}}
\multiput(100,30)(10,0){2}{\circle*{1}}
\multiput(105,25)(10,0){1}{\circle*{1}}
\multiput(110,20)(10,0){1}{\circle*{1}}

\multiput(30,60)(10,0){1}{\circle*{1}}
\multiput(35,55)(10,0){1}{\circle*{1}}
\multiput(30,50)(10,0){2}{\circle*{1}}
\multiput(35,45)(10,0){2}{\circle*{1}}
\multiput(30,40)(10,0){3}{\circle*{1}}
\multiput(35,35)(10,0){2}{\circle*{1}}
\multiput(30,30)(10,0){2}{\circle*{1}}
\multiput(35,25)(10,0){1}{\circle*{1}}
\multiput(30,20)(10,0){1}{\circle*{1}}

\multiput(50,80)(10,0){5}{\circle*{1}}
\multiput(55,75)(10,0){4}{\circle*{1}}
\multiput(60,70)(10,0){3}{\circle*{1}}
\multiput(65,65)(10,0){2}{\circle*{1}}
\multiput(70,60)(10,0){1}{\circle*{1}}

\multiput(70,50)(10,0){1}{\circle*{1}}
\multiput(65,45)(10,0){2}{\circle*{1}}
\multiput(60,40)(10,0){3}{\circle*{1}}
\multiput(65,35)(10,0){2}{\circle*{1}}
\multiput(70,30)(10,0){1}{\circle*{1}}

\multiput(70,20)(10,0){1}{\circle*{1}}
\multiput(65,15)(10,0){2}{\circle*{1}}
\multiput(60,10)(10,0){3}{\circle*{1}}
\multiput(55,5)(10,0){4}{\circle*{1}}
\multiput(50,0)(10,0){5}{\circle*{1}}

\put(30,15){\line(1,1){65}}
\put(45,0){\line(1,1){65}}

\put(95,0){\line(-1,1){24}}
\put(69,26){\line(-1,1){13}}
\put(54,41){\line(-1,1){24}}

\put(110,15){\line(-1,1){24}}
\put(84,41){\line(-1,1){13}}
\put(69,56){\line(-1,1){24}}

\put(235,40){\line(1,0){40}}
\put(275,40){\line(1,0){40}}
\put(275,40){\line(0,1){40}}
\put(275,0){\line(0,1){40}}

\put(235,40){\circle*{5}}
\put(275,40){\circle*{5}}
\put(315,40){\circle*{5}}
\put(275,0){\circle*{5}}
\put(275,80){\circle*{5}}

\end{picture} \hspace{15mm}
    } 
  \caption{$D$, $T(D)$, $D_2$ and $T(D_2)$ locally at a crossing 
  (shown without signs). }
  \label{f.t12}
  \end{figure}

Let $\T(D)$ be the Tait graph of our link diagram $D$. Our aim is to 
determine the Tait graph $\T(D_r)$ of the $r$-fold parallel $D_r$ 
of $D$ directly from $\T(D)$. Thus we need to find a construction of 
the graph $\T_r(D)$ from $\T(D)$ that completes the commutative 
diagram:

\begin{equation}\label{com}
\begin{array}{ccc}
D & \longrightarrow & \T(D)\\
\downarrow & & \downarrow \\
D_{r} & \longrightarrow & \T(D_{r})=\T_{r}(D) 
\end{array}.
\end{equation}
In order to achieve this we will start with $\T(D)$, construct 
$\T_{2}(D)$, and then find a general recurrence
relationship between $\T_{r+1}(D)$ and $\T_{r}(D)$.

For $\T_{2}(D)$ we need to define the overlay 
product of $\T$ and $\T^{\ast}$, denoted 
$\T\overlay\T^{\ast}$. In fact  the overlay product 
is defined for any pair of plane graphs $G$ and $H$ that have the 
property that the vertices of each lie in the faces of the other, so 
we give the general definition.

\begin{definition}
Let $G$ and $H$ be abstract graphs, and suppose they are plane, and
therefore equipped with embeddings $i:G\rightarrow S^{2}$ and 
$j:H\rightarrow S^{2}$. Suppose further that the vertices of each lie 
in the faces of the other. The {\em overlay product}, $G\overlay H$, 
is the image $i(G) \cup j(H)$ with vertices added wherever an edge of 
$G$ intersects an edge of $H$.

Every edge $e$ in $G\overlay H$ arises from an edge $e'$ in $G$ or 
$H$. If $e'$ has an edge weight, then $e$ is assigned the same edge 
weight as $e'$.
\end{definition}
In the case of the plane graph $G$ and its dual $G^{\ast}$, the 
embedding $j$ is induced from $i$, and the condition that the 
vertices of each lie in the faces of the other is satisfied by 
construction.

\setlength{\unitlength}{0.5pt}
\begin{figure}
\begin{picture}(150,100)(200,0)

\put(100,25){\circle*{5}}
\put(50,110){\circle*{5}}
\put(150,110){\circle*{5}}
\put(100,25){\line(-50,85){50}}
\put(100,25){\line(50,85){50}}
\qbezier{(50,110),(100,155),(150,110)}
\qbezier{(50,110),(100,65),(150,110)}

\put(250,110){\circle*{5}}
\put(250,60){\circle*{5}}
\put(250,0){\circle*{5}}
\put(250,60){\line(0,1){50}}
\qbezier{(250,0),(200,30),(250,60)}
\qbezier{(250,0),(300,30),(250,60)}
\put(290,110){\oval(80,110)[tl]}
\put(290,82.5){\oval(80,165)[r]}
\put(250,0){\line(1,0){40}}

\put(450,25){\circle*{5}}
\put(400,110){\circle*{5}}
\put(500,110){\circle*{5}}
\put(450,25){\line(-50,85){50}}
\put(450,25){\line(50,85){50}}
\qbezier{(400,110),(450,155),(500,110)}
\qbezier{(400,110),(450,65),(500,110)}
\put(450,110){\circle*{5}}
\put(450,60){\circle*{5}}
\put(450,0){\circle*{5}}
\put(450,60){\line(0,1){50}}
\qbezier{(450,0),(400,30),(450,60)}
\qbezier{(450,0),(500,30),(450,60)}
\put(490,110){\oval(80,110)[tl]}
\put(490,82.5){\oval(80,165)[r]}
\put(450,0){\line(1,0){40}}
\put(435.5,49.5){\circle*{5}}
\put(464.5,49.5){\circle*{5}}
\put(450,88){\circle*{5}}
\put(451,132.5){\circle*{5}}

\end{picture} 
\caption{For the Figure Eight knot, the graphs $\T$,  $\T^{\ast}$, and $\T_2=\T 
\overlay \T^{\ast}$}
\label{ThetaonThetastar}	    
\end{figure}

\begin{definition}
Let $\T$ be a Tait graph, then we define $\T_{r}$ inductively by $\T_1:=\T$ and 
\[\T_{r} := \T_{r-1}^{\ast} \overlay \T.\]
\end{definition}

We note that $\T_{r}$ can be defined in this way since, at each step, the vertices of 
$\T_{r}^{\ast}$ are the faces of $\T_{r}$, 
which are subsets of the faces of $\T$, and the vertices 
of $\T$ are also vertices of $\T_{r}$, and hence 
lie in faces of $\T_{r}^{\ast}$.


%

The following theorem shows that the overlay product is the construction required to complete the commutative diagram~\eqref{com}.
\begin{theorem}\label{t.cable} Let $D$ be a link diagram, then 
    $\T_{r}(D)=\T_{1}(D_{r}).$
\end{theorem}

\begin{proof}
We use induction on $r$. The base case, $r=2$, follows easily by considering a single crossing, as in Figure~\ref{f.t12}. 

Suppose  that
$    \T_{r-1}(D)=\T_{1}(D_{r-1})$.
Then
$   \T_{r}(D)
    =\left(\T_{r-1}(D)\right)^{*}\overlay\T_{1} 
    =\left(\T_{1}(D_{r-1})\right)^{*}\overlay\T_{1}$.
We need to show that
\begin{equation}
\label{chessboard}
   \left(\T_{1}(D_{r-1})\right)^{*}\overlay\T_{1}=\T_{1}(D_{r}).
\end{equation}
Consider a crossing in $D$. In $\T_{1}(D_{r-1})$ it yields an 
$r\times r$ chess board, with the long diagonal black. 
 In $\T_{1}(D_{r})$ there is a new row and column, so the 
chess board now has size $(r+1)\times (r+1)$. The long diagonal is 
still black, however.
The pattern of black squares in the $(r+1)\times (r+1)$ board can be 
obtained by taking the pattern of white squares in the $r\times r$ 
board, making them black, and then inserting a new long diagonal.
In terms of our graphs, this means taking the dual of 
$\T_{1}(D_{r-1})$, and then the overlay product with 
$\T_{1}(D)$ inserts the new long diagonal, giving Equation~\eqref{chessboard}. 
\end{proof}

The following two results show that the structure of the graphs 
$\T_{r}(D)$ is quite tightly constrained.
\begin{lemma}\label{l6}
Let $D$ be a link diagram. When $r$ is even all the faces of 
$\T_{r}(D)$ are squares. In the case when $r$ is odd, if $f$ is a 
face of $\T(D)$ and $f_r$ is the subgraph of $\T_{r}(D)$ that is 
contained in $f$, then $f_r$ contains a copy of $f$ and all other 
faces of $\T_{r}(D)$ are squares.
\end{lemma}

\begin{proof}
The statements are true for $r=1$ and $r=2$. We use induction on $r$. 
Suppose that the statements are true for all $r$ less than $k$.
Consider a face of $\T_{k+1}$. It is either
a face of $\T_{k}^{\ast}$, or
a part of a face of $\T_{k}^{\ast}$.
In the case where  a face of $\T_{k+1}$ is
a face of $\T_{k}^{\ast}$, we can immediately go back one more step. Each face of 
$\T_{k}^{\ast}$ corresponds to a vertex of $\T_{k}$, which 
must be either
(a) a face of $\T_{k-1}^{\ast}$;
(b) a new vertex of degree four formed by the overlay; or
(c) a vertex of $\T$.
In cases (a) and (b) we can deduce that the face of 
$\T_{k+1}^{\ast}$ must either be a face of $\T_{k-1}^{\ast}$ or a 
square. Case (c) does not arise, because if the 
face of $\T_{k}^{\ast}$ had come from a vertex of 
$\T$ then it would have been subdivided.

This leads us to the second case where   a face of $\T_{k+1}$ is
a part of a face of $\T_{k}^{\ast}$. A face of $\T_{k}^{\ast}$ may be 
subdivided in two ways. It either has a vertex of $\T$ in 
it, in which case it is subdivided into squares, or it has an edge of 
$\T$ crossing it. Here, too, it is subdivided into squares.
Hence the result. \end{proof}

\begin{figure}
\begin{picture}(50,50)(150,0)

\put(0,70){\circle*{5}}
\put(35,130){\circle*{5}}
\put(105,130){\circle*{5}}
\put(130,70){\circle*{5}}
\put(70,15){\circle*{5}}
\put(0,70){\line(35,60){35}}
\put(35,130){\line(1,0){70}}
\put(105,130){\line(25,-60){25}}
\put(130,70){\line(-60,-55){60}}
\put(70,15){\line(-70,55){70}}

\put(250,70){\circle*{5}}
\put(285,130){\circle*{5}}
\put(355,130){\circle*{5}}
\put(380,70){\circle*{5}}
\put(320,15){\circle*{5}}
\put(250,70){\line(35,60){35}}
\put(285,130){\line(1,0){70}}
\put(355,130){\line(25,-60){25}}
\put(380,70){\line(-60,-55){60}}
\put(320,15){\line(-70,55){70}}

\put(320,80){\circle*{5}}
\put(320,80){\line(0,1){50}}
\put(320,130){\circle*{5}}
\put(320,80){\line(20,9){47}}
\put(367,101){\circle*{5}}
\put(320,80){\line(9,-11){30.5}}
\put(350.5,43){\circle*{5}}
\put(320,80){\line(-9,-10){34.5}}
\put(285.5,42){\circle*{5}}
\put(320,80){\line(-21,9){51}}
\put(268.5,101.5){\circle*{5}}

\end{picture} 
\caption{A face of $\T$, and the corresponding faces of 
$\T \overlay \T^{\ast}$}
\label{squarefaces}	    
\end{figure}

The overlay product construction, together with Euler's equation, 
lead to the following formulae for the number of edges, faces and vertices of $T_r(D)$  which will be of use later.
\begin{lemma}\label{t.edges} Let $D$ be a link diagram, $f$, $e$, and $v$ be the number of faces, edges and vertices of
graph $\T(D)$, respectively, and let $f_{r}$, $e_{r}$, and $v_{r}$ be the 
number of faces, edges and vertices of $\T_{r}(D)$, respectively.
For $r$ even we have:
\[  e_{r} = r^{2}e, \quad\quad f_{r} = \frac{r^{2}e}{2},   \quad \quad v_{r} = 2+\frac{r^{2}e}{2};\]
and for $r$ odd: 
\[ e_{r} = r^{2}e, \quad \quad f_{r} = f+\frac{e(r^{2}-1)}{2},  \quad \quad v_{r} = v+\frac{e(r^{2}-1)}{2}. \]
\end{lemma}

\begin{proof}
If the original link diagram $D$ had $e$ crossings then the $r$-fold 
parallel of $D$ will have $r^{2}e$ crossings. Since the number of 
edges of $\T_{r}$ is equal to the number of crossings of the link 
$D_{r}$, we have
$e_{r}=r^{2}e$. 

If $r$ is even all the faces are squares, and so
$4f_{r}=2e_{r}$,
and hence
$f_{r}=\frac{r^{2}e}{2}$.
Now Euler's equation gives
$v_{r}=2+\frac{r^{2}e}{2}$.

If $r$ is odd we count the faces in more detail. Let $f_{r,n}$ be the 
number of faces in $\T_{r}$ bounded by $n$ edges. Then, by Lemma~\ref{l6}, 
\[ f_{r}
 = x + f_{1,1} + f_{1,2} + f_{1,3} + f_{1,4} + f_{1,5} + \dots 
 = x + f, \]
for some $x$. Now let us count the edges.
\[ 2e_{r}=f_{r,1}+ 2f_{r,2}+3f_{r,3}+4f_{r,4}+5f_{r,5}+\dots 
 =4x+ f_{1,1} +2f_{1,2}+3f_{1,3}+4f_{1,4}+5f_{1,5}+\dots 
 =4x+2e.\]
Therefore
$x=\frac{e(r^{2}-1)}{2}$,
and so
$f_{r}=f+\frac{e(r^{2}-1)}{2}$.
Finally, we use Euler's equation again to deduce that
$v_{r}=v+\frac{e(r^{2}-1)}{2}$, as required. 
\end{proof}

We conclude this section with some notation and a technical  result that we require later.
Given any plane graph $G$ with $r$-fold overlay $G_r$, there is a 
natural map
\[\varphi_{r}:E(G_r)\rightarrow E(G_{r-1}{}^*)\cup E(G)\]
induced by the definition of the overlay product: the 
edge $e\in E(G_r)$ lies in an edge 
$\varphi_{r}(e)$ of $G_{r-1}{}^*$ or $G$. We call the map 
$\varphi_{r}$ the {\em projection} of $E(G_r)$ onto 
$E(G_{r-1}{}^*)\cup E(G)$.

When the link diagram $D$ is not alternating, its Tait graph 
$\T(D)$ will be signed. Evidently, given an edge 
$e\in\T(D)$, the corresponding edge $e^{*}\in\T(D)^{*}$ will have the 
opposite sign.
The sign of an edge $e\in E(G_r)$ is the same as the sign of 
$\varphi_{r}(e)$, whether this is in $E(G_{r-1}{}^*)$ or $E(G)$. Of 
course, if $\varphi_{r}(e)\in E(G_{r-1}{}^*)$ then the sign of $e$ is 
opposite to the sign of $\varphi_{r}(e)^{*}\in E(G_{r-1})$.
The following lemma tells us that if we take two adjacent edges in 
$G_r$ with the property that one edge arises from an edge in 
$G_{r-1}{}^*$ and the other from an edge in $G$, then the two edges 
are of different signs.
\begin{lemma}
\label{differentsigns}
Let $G$ be a signed plane graph, $G_r$ denote its $r$-fold overlay, 
and $\varphi_{r}$ be the projection of $E(G_r)$ onto 
$E(G_{r-1}{}^*)\cup E(G)$. If $e\in\varphi_{r}^{-1}(E(G_{r-1}{}^*))$ 
and $f\in\varphi_{r}^{-1}(E(G))$ are adjacent, they have different 
signs.
\end{lemma}
\begin{proof}
Because $\varphi_{r}(e)\in E(G_{r-1}{}^*)$ we must have
$\varphi_{r}(e)^{*}\in E(G_{r-1})$.
Now $\varphi_{r-1}(\varphi_{r}(e)^{*})\in E(G)$, because $e$ and $f$ 
are adjacent in $G_{r}$. In fact,
$\varphi_{r-1}(\varphi_{r}(e)^{*})=\varphi_{r}(f)$.
Therefore $\varphi_{r}(e)^{*}$ has the same sign as $f$, which 
implies that $\varphi_{r}(e)$ has the opposite sign to $f$, and hence 
the result.
\end{proof}

\subsection{The genus of the ribbon graphs of $r$-fold parallels}
In this subsection we determine how the genus of a ribbon graph of a 
link diagram relates to the genus  of the ribbon graphs of its 
parallels.

Let $D$ be a link diagram and $D_r$ be its $r$-fold parallel. In 
forming the parallel, every crossing $c$ of $D$ gives rise to a set, 
$C_r(c)$, of crossings of $D_r$. If $s$ is a state of $D$ then we can 
form a state $s_r$ of $D_r$ by, for each crossing $c$, splicing all 
of the crossings in $D_r$ in the same way as $c$ was spliced in the 
state $s$. So if a crossing $c$ of $D$ is $A$-spliced (respectively 
$B$-spliced) in the state $s$, then every crossing of  in $C_r(c)$ is 
$A$-spliced (respectively $B$-spliced) in the state $s_r$ of $D_r$.



\begin{theorem}\label{ca1}
Let $D$ be a link diagram and $s$ be a state of $D$. Then 
\[ g(\G(D_{r+1},s_{r+1}))=(r+1)\cdot g(\G(D,s))+r^2\cdot 
e(\G(D,s))-r. \]
\end{theorem}
The proof of this theorem will follow from Theorem~\ref{ca4} below.

Theorem~\ref{ca1} has immediate applications to the special ribbon 
graphs of a link diagram described above.
\begin{corollary}\label{ca2}
Let $D$ be a link diagram and $D_{r+1}$ be its $(r+1)$-fold parallel. 
Then
\begin{enumerate} 
\item $g(\A(D_{r+1}))= (r+1)\cdot g(\A(D) ) +r^2\cdot e(\A(D))-r$;
\item $g(\B(D_{r+1}))= (r+1)\cdot g(\B(D) ) +r^2\cdot e(\B(D))-r$;
\item $g(\s(D_{r+1}))= (r+1)\cdot g(\s(D) ) +r^2\cdot e(\s(D))-r$.
\end{enumerate}
\end{corollary}
\begin{proof}
For the first item, suppose $s$ is obtained by choosing an 
$A$-splicing at each crossing. Then $s_{r+1}$ is also obtained by 
choosing an $A$-splicing at each crossing. We have $\G(D,s)=\A(D)$ 
and $\G(D_{r+1},s_{r+1})=\A(D_{r+1})$, and the result follows from 
Theorem~\ref{ca1}. The second item is proved similarly.

As for the third item, note that if $c$ is a crossing in $D$ with 
oriented sign $+$ (respectively $-$) then every crossing in  
$C_{r}(c)$ also has oriented sign $+$ (respectively $-$). 
Thus if $s$ is a state of $D$ such that $\G(D,s)=\s(D)$, it follows 
that $s_{r+1}$ is a state of  $D_{r+1}$ such that 
$\G(D_{r+1},s_{r+1})=\s(D_{r+1})$. The result then follows from 
Theorem~\ref{ca1}.
\end{proof}

The Turaev genus, $g_t(L)$, (see \cite{Ab09,Lo08,Tu97}) of a link $L$ 
is defined as the minimum genus $g(\A(D))$, over all link diagrams 
$D$ of $L$. The following corollary provides a bound on the  Turaev 
genus of an $r$-fold parallel of a link in terms of the Turaev genus 
of the original link.
\begin{corollary} Let $L$ be a link and $L_r$ be its $r$-fold 
parallel. Then an upper bound on the Turaev genus, $g_t(L_r)$, of 
$L_r$ is 
\[ g_t(L_r)\leq (r+1)\cdot g_t(L)+r^{2}c-r, \]
where $c$ is the crossing number of any diagram $D$ for $L$ attaining 
the Turaev genus.
\end{corollary}
\begin{proof}
The corollary follows immediately from the definition of the Turaev 
genus and Corollary~\ref{ca2}.
\end{proof}

Theorem~\ref{ca1} also provides a way to calculate the genus of a 
ribbon graph of $D_r$ in terms of the Tait graph of $D$.
\begin{corollary}\label{ca3} 
Let $D$ be a link diagram, $\T=\T(D)$ be 
its Tait graph, and $s$ be a state of $D$. Then if $A$ is the set of 
edges of $\T$ such that $\T^A= \G(D,s)$,
\[2g(\G(D_{r+1},s_{r+1}))=(2r^2+r+1)\cdot e(\T) - (r+1) \cdot(p(\T-A^c) - 
p(\T-A)) + 2.\] 
\end{corollary}
\begin{proof}
By Theorem~\ref{ca1},
\[ g(\G(D_{r+1},s_{r+1}))=(r+1)\cdot g(\G(D,s))+r^{2}\cdot 
e(\A(D))-r. \]
Since $\G(D,s)=\T^A$, Proposition~\ref{p.pd2} gives 
\[g(\G(D,s) ) = 1 + e(\T)/2 - p(\T-A^c)/2 - p(\T-A)/2, \]
and so 
\[g(\G(D_{r+1},s_{r+1}))=(2r^2+r+1)/2\cdot e(\T) - (r+1)\cdot p(\T-A^c)/2 - 
(r+1) \cdot p(\T-A)/2 + 1. \]
\end{proof}

\bigskip

We will now turn our attention to the proof of Theorem~\ref{ca1}. In 
fact, we will prove the following more general result which expresses 
$g(\A(D_{r+1}))$ in terms of the genus of lower order parallels of $D$.
Theorem~\ref{ca1} follows by repeated application of this result.
\begin{theorem}\label{ca4}
Let $D$ be a link diagram and $s$ be a state of $D$. Then for each  
$r\in \mathbb{N}$ we have 
\[ g(\G(D_{r+1}),s_{r+1})=g(\G(D_{r},s_r))+g(\G(D,s))+r\cdot 
e(\G(D,s))-1. \]
\end{theorem}

\medskip

\begin{figure}
\begin{tabular}{l  |c | c }
Crossing $c$ & 
\raisebox{-5mm}{\includegraphics[height=1.2cm]{taitplus}} & 
\raisebox{-5mm}{\includegraphics[height=1.2cm]{taitminus}}
\\ \hline
\raisebox{3mm}{edge $e$ in $\T$} &  \includegraphics[width=2cm]{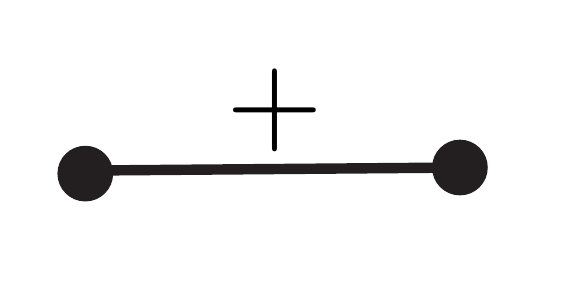} 
& \includegraphics[width=2cm]{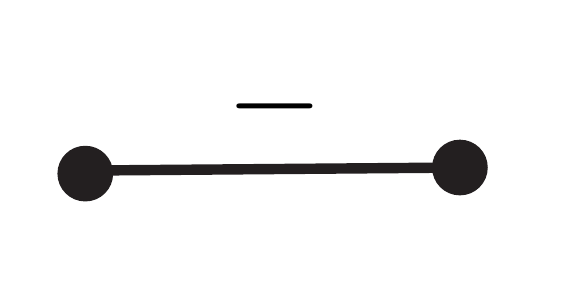}
\\ \hline
Splicing & \begin{tabular}{c|c}
\raisebox{15mm}{~} \includegraphics[height=1.2cm]{splice2} &  
\includegraphics[height=1.2cm]{splice3}
\\ $A$-splicing &  $B$-splicing
\end{tabular}
& \begin{tabular}{c|c}
\raisebox{15mm}{~} \includegraphics[height=1.2cm]{splice2} &  
\includegraphics[height=1.2cm]{splice3}
\\ $B$-splicing &  $A$-splicing 
\end{tabular}\\
\hline
$e\in A$?  & Yes \hspace{15mm} No   & Yes \hspace{15mm} No  
\end{tabular}

\caption{The correspondence between splicings in $D$ and edges in 
$\T(D)$
}
\label{f.s}
\end{figure}

To prove the theorem we begin by discussing some relations of  
$\G(D_{r+1},s_{r+1})$ with partial duals. Suppose that $D$ is a link 
diagram, $s$ is a state of $D$, and $\T(D)$ is its Tait graph. Since 
the ribbon graphs of a link diagram   are all partial duals of the 
Tait graph, by Proposition~\ref{p.taitdual}, we know that there exists a unique set of edges 
$A\subseteq E(\T(D))$, such that $\G(D,s)=\T(D)^A$ (the uniqueness 
uses the fact that the edges are signed), and similarly, for each 
$r\in \mathbb{N}$,  there is a subset of edges  $A_r\subseteq 
E(\G(D_r, s_r))$ such that $\G(D_r, s_r) = \T(D_r)^{A_r} = 
(\T_r(D))^{A_r}$. Our first aim is to find a  description of the set 
$A_r$ in terms of the set $A$.

In the state $s_r$ of $D_r$, every crossing in $C_r(c)$ (which is the 
set of crossings in $D_r$ arising from the crossing $c$ in $D$) is 
spliced in the same way as the crossing $c$ of $D$. Now let $e$ be 
the edge of $\T(D)$ corresponding to $c$, and $E_r(e)$ be the set of 
edges of $D_r$ corresponding to the crossings in $C_r(c)$. Then from 
Figure~\ref{f.s}, we see that $e$ is or is not in $A$ depending on 
the splicing at $c$. Now, let $e_r$ be an edge in $E_r(e)$, and 
$c_r\in C_r(c)$ be the crossing corresponding to $e_r$. Then $c_r$ 
and $c$ are  spliced in the same way in the states $s_r$ and $s$ 
respectively. So to determine if $e_r\in A_r$, which is the set of 
edges of $ \T(D_r)$ that describe the state $s_r$, we consult 
Figure~\ref{f.s} and see that $e_r\in A_r$ if and only if
\begin{itemize}
\item $e$ and $e_r$ have the same sign and $e\in A$; or
\item $e$ and $e_r$ have different signs and $e\notin A$. 
\end{itemize}  
Thus the set of edges $A_r$ such that $\G(D_r, s_r) = \T(D_r)^{A_r}$ 
is constructed by, for each edge $e_r$ of $\T(D_r)$, looking at the 
edge $e$ in $\T(D)$ that is associated with $e_r$, and putting 
$e_{r}$ in $A_r$ if $e$ and $e_r$ have the same sign and $e\in A$, or 
if $e$ and $e_r$ have different signs and $e\notin A$. So,
\begin{multline}
\label{e.ar}
A_r=\{ e \in E(\T(D)_r) \; |\;  [  \rho 
(e) \in A \text{ with same sign as  }e]  \text{ or } [  
\rho (e) \notin A \text{ with different sign to   }e ] \} , 
\end{multline}
where $\rho (e)$ is the edge in $\T(D)$ such that $e\in 
E_r(\rho(e))$.

This discussion in summarized by the following lemma.
\begin{lemma}\label{ca6}
Let $D$ be a link diagram and $s$ be a state of $D$. If $ \G(D,s)= 
\T(D)^A$, then  $\G(D_r, s_r)=\T(D_r)^{A_r}$, where $A_r$ is given by 
Equation~\eqref{e.ar}.
\end{lemma}

We will also need the following technical lemma.
\begin{lemma}\label{ca5}
Let $\T$ be a Tait graph and, for each $k\in \mathbb{N}$,  $A_k$ be 
given by Equation~\eqref{e.ar}.
Then
\begin{enumerate}
\item \label{ca5.2}  $p(\T_{r+1} - A_{r+1}) = p( \T_r{}^* - A_r{}^* 
)+ p(\T-A)$;
\item \label{ca5.3} $p(\T_{r+1} - (A_{r+1})^c) = p( \T_r{}^* - 
(A_r{}^*)^c )+ p(\T-A^c)$.
\end{enumerate}       
\end{lemma}

\begin{proof}
For the first item, let $\varphi$ be the  projection of   
$E(G_{r+1})$ onto $ E(G_{r}{}^*) \cup E(G)$, as described at the end 
of Subsection~\ref{ss.rfold}. In addition, let $H$ denote the plane 
subgraph of $G_{r+1} = G_{r}{}^* \overlay G$ induced by 
$\varphi^{-1}(E(G) )$, and let $K$ denote the plane subgraph of 
$G_{r+1} = G_{r}{}^* \overlay G$ induced by 
$\varphi^{-1}(E(G_{r}{}^*) )$. Then $H\cup K = G_{r+1}$ and $H\cap K$ 
is the set of vertices created by the overlay  $G_{r}{}^* \overlay 
G$.
 
Let $v\in H\cap K$. Then $v$ is incident to exactly four edges. Call 
these edges $e,e',f,f'$, two of which are in $H$ and two of which are 
in $K$. Suppose that $e$ and $e'$ are in $H$, and that $f$ and $f'$ 
are in $K$.
 
By Lemma~\ref{differentsigns}, $e$ and $e'$ are both of the same 
sign, $m$ say, and $f$ and $f'$ are both of the same sign $-m$. It 
then follows from the definition of $A_{r+1}$, that either $e,e' \in 
A_{r+1}$, or $f,f' \in A_{r+1}$. Consequently no connected 
component of $\T_{r+1} - A_{r+1}$ has edges in both $H$ and $K$. So 
every connected component of $\T_{r+1} - A_{r+1}$ is a connected 
component of exactly one of the plane graphs $H-A_{r+1}$ and 
$K-A_{r+1}$. Therefore
\[  p(\T_{r+1} - A_{r+1}) = p(H- A_{r+1}) + p(K- A_{r+1}). \]
Finally, by the definition of the overlay product, we have that $p(H 
- A_{r+1}) = p(\T-A)$ and that $p(K- A_{r+1}) = p(\T_r{}^* - 
A_r{}^*)$, and the result follows.
 
The proof of the second item is similar.
\end{proof}

We can now prove Theorem~\ref{ca4}
\begin{proof}[Proof of Theorem~\ref{ca4}]
Let $\T:=\T(D)$ and $\T_k:=\T(D_k)$ be the Tait graphs of $D$ and $D_k$ 
respectively. In addition let $s$ be a state of $D$ and $A\subseteq 
E(\T)$ be such that $\G(D,s)=\T^A$. Then by Theorem~\ref{t.cable}, 
$\T(D_k) = \T_k$, the $k$-fold overlay of $\T$. Moreover, by 
Lemma~\ref{ca6}, $\G(D_k,s_k) = \T_k{}^{A_k}$, where $A_k$ is given 
by  Equation~\eqref{e.ar}.
We then have
\[ 2g(\G(D_{r+1},s_{r+1})) =  2g(\T_{r+1}{}^{A_{r+1}}) = 2 + e( 
\T_{r+1}{}^{A_{r+1}}) - p(\T_{r+1} - A_{r+1}{}^c) -  p(\T_{r+1} - 
A_{r+1}), \]
where the second equality follows from Proposition~\ref{p.pd2}.
Using Lemma~\ref{ca5} to expand the above expression, we have
\[  2g(\G(D_{r+1},s_{r+1}))  = 2 + e(\T_{r+1}{}^{A_{r+1}}) - p(\T_r{}^* - 
(A_r{}^*)^c) - p(\T-A^c) - p(\T_r{}^* - A_r{}^*) - p(\T-A). \]
Now, since partial duality does not change the number of edges, we can use  Lemma~\ref{t.edges} twice to get
\[e(\T_{r+1}{}^{A_{r+1}}) =e(\T_{r+1}) = (r+1)^2 e(\T) = (2r+1)e(\T) + e(\T_r) = 
(2r+1)e(\T) + e(\T_r{}^*). \]
Thus we can write
\begin{multline*}
2g(\G(D_{r+1},s_{r+1})) = \left[2r\cdot e(\T) - 2\right] + \left[2 + 
e(\T_r{}^*) - p(\T_r{}^* - (A_r{}^*)^c) -  p(\T_r{}^* - A_r{}^*) 
\right]
\\ +\left[2 + e(\T) - p(\T-A^c) - p(\T-A) \right].
\end{multline*}
Which, by Proposition~\ref{p.pd2} gives
\begin{multline*}
2g(\G(D_{r+1},s_{r+1})) = 2g((\T_r{}^*)^{A_r{}^*}) + 2g(\T^A) + 
2r\cdot e(\T) - 2
\\= 2g((\T_r)^{A_r}) + 2g(\T^A) +2r\cdot e(\T) -2
\\= 2g(\G(D_{r},s_r)) + 2g(\G(D,s)) + 2r\cdot e(\G(D,s))-2,
\end{multline*}
as required.
\end{proof}


\begin{thebibliography}{99}

\bibitem{Ab09} T. Abe, The Turaev genus of an adequate knot. Topology 
Appl. 156 (2009), no. 17, 2704--2712.




\bibitem{Bondy} A. Bondy and U. Murty, Graph theory, Graduate Texts 
in Mathematics, 244. Springer-Verlag, New York, 2008.

\bibitem{CKS07}  A. Champanerkar, I. Kofman  and N. Stoltzfus, Graphs 
on surfaces and Khovanov homology. Algebr. Geom. Topol. 7 (2007), 
1531--1540.



\bibitem{CP} S. Chmutov and I. Pak, The Kauffman bracket of virtual 
links and the Bollob\'as-Riordan polynomial. Moscow Mathematical 
Journal {\bf 7} (2007) 409-418, {\tt arXiv:math.GT/0609012}.

\bibitem{CV} S. Chmutov, J. Voltz, Thistlethwaite's theorem for 
virtual links. To appear in J. Knot Theory Ramifications {\bf 17}  
(2008), 1189-1198, {\tt arXiv:0704.1310}.

\bibitem{Ch1} S. Chmutov, Generalized duality for graphs on surfaces 
and the signed Bollob\'as-Riordan polynomial. J. Combin. Theory 
Series B, {\bf 99} (3) (2009), 617-638  {\tt arXiv:0711.3490}.

\bibitem{Cr04} P. Cromwell,  Knots and links. Cambridge University 
Press, Cambridge, 2004.

\bibitem{Detal} O. T. Dasbach, D. Futer, E. Kalfagianni, X.-S. Lin, 
N. W. Stoltzfus, The Jones polynomial and graphs on surfaces. 
Journal of Combinatorial Theory Series B, {\bf 98} (2) (2008), 
384-399, {\tt arXiv:math.GT/0605571}.

\bibitem{Detal2} O. T. Dasbach, D. Futer, E. Kalfagianni, X.-S. Lin, 
N. W. Stoltzfus, Alternating sum formulae for the determinant and 
other link invariants. J. Knot Theory Ramifications, in press, {\tt 
arXiv:math/0611025}.

\bibitem{DL10} O. Dasbach and A. Lowrance, Turaev genus, knot 
signature, and the knot homology concordance invariants. Preprint, 
{\tt arXiv:1002.0898}.






\bibitem{EMM} J. A. Ellis-Monaghan and I. Moffatt, Twisted duality 
and polynomials of embedded graphs. Preprint, {\tt arXiv:0906.5557}.

\bibitem{FKP08} D. Futer, E. Kalfagianni, and J. Purcell, Dehn 
filling, volume, and the Jones polynomial. J. Differential Geom. 78 
(2008), number 3, 429--464.

\bibitem{FKP09} D. Futer, E. Kalfagianni, and J. Purcell, Symmetric 
links and Conway sums: volume and Jones polynomial. Math. Res. Lett. 
16 (2009), number 2, 233--253.





\bibitem{HM11} S. Huggett and I. Moffatt, Bipartite partial duals and 
circuits in medial graphs. Preprint.










\bibitem{KRVT09} T. Krajewski, V. Rivasseau and F. Vignes-Tourneret, 
Topological graph polynomials and quantum field theory, Part II: 
Mehler kernel theories, Ann. Henri Poincar\'e, 12, (2011),1-63, {\tt arXiv:0912.5438}.



\bibitem{Lo08} A. Lowrance, On knot Floer width and Turaev genus. 
Algebr. Geom. Topol. 8 (2008), no. 2, 1141--1162.




\bibitem{Mo2} I. Moffatt,  Unsigned state models for the Jones 
polynomial, Ann.  Comb., 15 (2011) 127-146. {\tt 
arXiv:0710.4152}.

\bibitem{Mo3} I. Moffatt,  Partial duality and Bollob\'as and 
Riordan's ribbon graph polynomial. Discrete Math. {\bf 310} (2010) 
174-183, {\tt arXiv:0809.3014}.

\bibitem{Mo4} I. Moffatt, A characterization of partially dual 
graphs. To appear in J. Graph Theory, {\tt arXiv:0901.1868}.

\bibitem{Mo5} I. Moffatt,  Partial duals of plane graphs, 
separability and the graphs of knots. Preprint, {\tt 
arXiv:1007.4219}.







\bibitem{Tu97}V. Turaev, A simple proof of the Murasugi and Kauffman 
theorems on alternating links. Enseign. Math.  33  (1987) 
203Ð225.


\bibitem{VT10}  F. Vignes-Tourneret, Non-orientable quasi-trees for 
the Bollob\'as-Riordan polynomial. European J. Combin., 32, (2011) 510-532.

\bibitem{We93} D.J.A. Welsh, Complexity: Knots, Colorings and 
Counting. Cambridge Univ. Press, Cambridge, 1993.


\bibitem{Wi09} T. Widmer, Quasi-alternating Montesinos links. J. Knot 
Theory Ramifications 18 (2009), 1459-1469.


\end{thebibliography}
\end{document}